\documentclass{article}

\usepackage{arxiv}

\usepackage[utf8]{inputenc} % allow utf-8 input
\usepackage[T1]{fontenc}    % use 8-bit T1 fonts
\usepackage{hyperref}       % hyperlinks
\usepackage{url}            % simple URL typesetting
\usepackage{booktabs}       % professional-quality tables
\usepackage{amsfonts}       % blackboard math symbols
\usepackage{nicefrac}       % compact symbols for 1/2, etc.
\usepackage{microtype}      % microtypography
\usepackage{amsthm}
\usepackage{mathtools}
\usepackage{amssymb}
\usepackage{todonotes}

\usepackage[shortlabels]{enumitem}
\usepackage{color}
\usepackage{comment}
\usepackage{subcaption}
\usepackage{graphicx}

\newtheorem{theorem}{Theorem}

\newtheorem{lemma}[theorem]{Lemma}
\newtheorem{remark}[theorem]{Remark}
\newtheorem{assumption}{Assumption}

\numberwithin{theorem}{section}

\title{Sensitivity of SDE Solutions to Perturbations of the Diffusion and Drift}
\author{Jeremiah Birrell\\
  Department of Mathematics\\
  Texas State University\\
  San Marcos, TX,  USA \\
  \texttt{jbirrell@txstate.edu}}

\begin{document}

\maketitle

\begin{abstract}
 We develop a method for bounding the sensitivity of solutions to  stochastic differential equations (SDEs) to     changes  in the     drift, $F$, and diffusion, $\sigma$, by using a combination of information-theoretic uncertainty quantification bounds,  functional inequalities, and   judiciously chosen coupled auxiliary SDEs. The method is capable of producing  non-asymptotic bounds which are well behaved in the $T\to \infty$ limit and does not require the perturbations to $F$ and $\sigma$ to be small. Our approach applies to  expectations of both time-averaged and exponentially discounted observables and also produces sensitivity bounds for linear parabolic  PDEs.  When applied to stationary solutions and Lipschitz observables, our results  produce  bounds on the $1$-Wasserstein distance between    invariant measures   which have optimal scaling in each error term. The present method significantly expands on prior information-theoretic SDE sensitivity bounds, which are only applicable to perturbations of the drift. 
\end{abstract}

\section{Introduction}
Consider two stochastic differential equations (SDEs) on $\mathbb{R}^d$,
\begin{align}\label{eq:(X,Y)}
dX_t=&F(t,X_t)dt+\sigma(t,X_t)dW_t\,,\\
dY_t=&G(t,Y_t)dt+\eta(t,Y_t)dW_t\,,\notag
\end{align}
driven by a $d$-dimensional Wiener process $W_t$; we will view $X_t$ as the baseline model and $Y_t$ as a perturbation.  Our goal is to compare observables  (i.e., measurable functionals of the paths) of the two processes, where we allow  both the  diffusion and drift   to be perturbed.  Previous  information-theoretic uncertainty quantification (UQ) methods    \cite{li2012computation,chowdhary2013distinguishing,dupuis2016path,doi:10.1137/15M1047271,katsoulakis2017scalable,gourgoulias2020biased,dupuis2020sensitivity,Birrell2020,birrell2021quantification,birrell2025concentration} are effective at producing sensitivity bounds for stochastic systems  that are well-behaved as $t\to\infty$ (e.g., when applied to time averages), however they rely on information-theoretic divergences (e.g., the KL-divergence) which require absolute continuity to produce non-trivial results. Allowing  $X_t$ and $Y_t$ to have differing diffusion terms, $\eta\neq \sigma$,   implies that the distributions of their solutions on path-space are not absolutely continuous, thus preventing us from bounding the difference in expectations under $X_t$ and $Y_t$ via a direct application of these previous information-theoretic methods.  In this work we develop an approach which circumvents this issue via the introduction of   auxiliary SDEs coupled to $X_t$ and $Y_t$, followed by  a combination of   stochastic calculus estimates, functional inequalities, and    information-theoretic  UQ bounds.

Under appropriate ergodicity assumptions on the baseline model $X_t$ (in the form of $\log$-Sobolev inequalities) our results  imply  uniform-in-time sensitivity bounds on the time-average of a function $h(t,x)$.  Specifically, in Theorem \ref{thm:bounds_time_avg} we obtain bounds of the form
\begin{align}\label{eq:time_avg_bounds_preview}
   \sup_{T>0}\left|    \mathbb{E}\left[\frac{1}{T}\int_0^T h(t,Y_t)dt\right]-  \frac{1}{T}\int_0^T E_{\mu_X^*}[h_t] dt \right|=O\left(\|F-G\|_\infty+\|\sigma-\eta\|_\infty\right) \,,
\end{align}
where $h_t\coloneqq h(t,\cdot)$, 
assuming that the drift and diffusion for $X_t$ are time-homogeneous and the process is started in an invariant distribution, $X_0\sim\mu_{X}^*$.  In particular, \eqref{eq:time_avg_bounds_preview} implies uniform convergence of time averages as the perturbations approach zero.     We note that if one is interested in initial conditions other than the invariant distribution, our results can be combined with well-established methods for proving convergence to the invariant distribution.   

  If  the drift and diffusion for both $X_t$ and $Y_t$ are time-homogeneous and they are started in  invariant distributions, $\mu_{X}^*$ and  $\mu_{Y}^*$ respectively,  then  in Theorem \ref{thm:Wasserstein}  we   derive a bound on the $1$-Wasserstein distance between these invariant distributions,
\begin{align}\label{eq:1_Wasserstein_bounds_preview}
    W_1(\mu_{Y}^*,\mu_{X}^*)=O\left(\|F-G\|_\infty+\|\sigma-\eta\|_\infty\right) \,.
\end{align}
By considering Ornstein–Uhlenbeck (OU) processes,   we show that the linear scaling of \eqref{eq:1_Wasserstein_bounds_preview} in   $\|F-G\|_\infty$   and in  $\|\sigma-\eta\|_\infty$ are both optimal.

Our method can also be used to obtain sensitivity bounds on exponentially discounted observables of time-inhomogeneous SDEs, which are of interest in applications to   economics, see, e.g.,   \cite{page2013applications,peter2014uncertainty},  and control theory, see, e.g., \cite{kushner2013numerical}.  Specifically, in Theorem \ref{thm:bounds_discounted_obs}  we derive bounds of the form
\begin{align}\label{eq:discounted_cost_bounds_preview}
\left|    \mathbb{E}\left[\int_0^\infty h(t,Y_t)  e^{-rt}dt\right]-\mathbb{E}\left[\int_0^\infty h(t,X_t)  e^{-rt}dt\right]\right|=O\left(\|F-G\|_\infty+\|\sigma-\eta\|_\infty\right) \,.
\end{align}

Finally, our method produces sensitivity bounds for  solutions to  linear parabolic  PDEs  of the form
\begin{align}\label{eq:Feynman_Kac_pde_intro}
&\partial_t u(t,x)=-A[u_t](t,x)-h(t,x)\,,\,\,\,\,u(T,x)=f(x)\,,\\
&A[u](t,x)\coloneqq\frac{1}{2}\sum_{i,j,k}\sigma^i_k(t,x)\sigma^j_k(t,x)\partial_{i}\partial_j u(x)+\sum_iF^i(t,x)\partial_iu(x)\,.\notag
\end{align}
More specifically, if $u(t,x)$ solves \eqref{eq:Feynman_Kac_pde_intro} and $v(t,x)$ solves \eqref{eq:Feynman_Kac_pde_intro} but with $G$ in place of $F$  and $\eta$ in place  $\sigma$ then,  in Theorem \ref{thm:parabolic_PDE}, we derive $L^1$ error bounds of the form
\begin{align}\label{eq:PDE_bounds_preview}
\|v(t,\cdot)-u(t,\cdot)\|_{L^1(\mu_X^*)}=  O\left(\|F-G\|_\infty+\|\sigma-\eta\|_\infty\right) 
\end{align}
 in the case where the SDE for $X_t$ is time-homogeneous and has an invariant distribution, $\mu_X^*$.

We provide an outline of the key ideas of our technique in Section \ref{sec:proof_outline}.  Applications of these results are explored in  Section \ref{sec:applications}. Details  proofs of the underlying estimates can be found in the Section \ref{sec:proofs}.
While we have previewed our results in the case of bounded perturbations, we emphasize that our method can still produce meaningful bounds  when $\|F-G\|_\infty=\infty$.  In addition, we emphasize that our theorems are non-asymptotic in nature; below we provide explicit non-asymptotic formulas for the bounds.     Our results do require $\sigma-\eta$ to be bounded and $\sigma$ to be invertible; thus the approach developed here   requires the baseline model to be a uniformly elliptic diffusion.  Generalizing beyond these assumptions is a task we leave for future work.

\subsection{Related Work}\label{sec:related}
The present work   is related to a large body of research on perturbation bounds, also called sensitivity or condition number bounds, for Markov processes and their invariant measures.  The majority of prior studies have focused on Markov chains on a discrete state-space \cite{schweitzer1968perturbation,meyer1980condition,haviv1984perturbation,funderlic1986sensitivity,seneta1988perturbation,kirkland1998applications,cho2000markov,cho2001comparison,Thiede_2015,froyland2016stability,seneta2021sensitivity},  but some authors have also studied  perturbation bounds  for Markov chains on a countable state-space \cite{rabta2008strong,cruz2020sensitivity} or a  general Polish state space \cite{rudolf2018perturbation},  and for continuous time, countable state-space Markov processes \cite{liu2012perturbation,zeifman2014perturbation}. Of particular relevance to this work are the approaches to sensitivity analysis for SDEs with perturbed drift \cite{huggins2017quantifying,Birrell2020,vladimirov2023entropy} and especially the aforementioned information-theoretic uncertainty quantification (UQ) approach to sensitivity bounds for stochastic systems \cite{li2012computation,chowdhary2013distinguishing,dupuis2016path,doi:10.1137/15M1047271,katsoulakis2017scalable,gourgoulias2020biased,dupuis2020sensitivity,Birrell2020,birrell2021quantification,birrell2025concentration}.  However,  to the best of the author's knowledge, sensitivity bounds  of the type obtained here  for SDEs with perturbed diffusion  were out of reach of established techniques.  Thus this work constitutes a novel method for obtaining sensitivity bounds on  SDEs with simultaneously perturbed diffusion and drift.

\section{Outline of the Proof} \label{sec:proof_outline}
In this section we provide  an outline of the key ideas underpinning our technique for deriving bounds of the form \eqref{eq:time_avg_bounds_preview},  \eqref{eq:1_Wasserstein_bounds_preview}, and \eqref{eq:discounted_cost_bounds_preview}. 
The key new tool that enables   consideration of perturbations to the diffusion is an appropriately defined   auxiliary process $\tilde X_t$ that has the same diffusion term as $X_t$, and is constructed so that    observables of $\tilde X_t$  are close to those of both  $X_t$ and $Y_t$.   More specifically, we introduce  two auxiliary processes  $\tilde Y_t$ and   $\tilde X_t$, coupled to $X_t$ and $Y_t$   respectively,  as follows:
\begin{align}
dX_t=&F(t,X_t)dt+\sigma(t,X_t)dW_t\,,\label{eq:X}\\
    d\tilde Y_t=&\left(G(t,\tilde Y_t)+\eta(t,\tilde Y_t)\sigma^{-1}(t,X_t)\left(F(t,X_t)-G(t,X_t)+\kappa(X_t-\tilde Y_t)\right)\right)dt +\eta(t,\tilde Y_t)dW_t \,,\label{eq:tilde_Y}\\
{}\notag\\
      d\tilde X_t=&\left(G(t,\tilde X_t)-\kappa(\tilde X_t-Y_t)\right)dt+\sigma(t,\tilde X_t)dW_t \,,\label{eq:tilde_X}\\
    dY_t=&G(t,Y_t)dt+\eta(t,Y_t)dW_t\,, \label{eq:Y}
\end{align}
where   the coupling strength, $\kappa\in[0,\infty)$, will be chosen later. As systems of SDEs, $(X_t,\tilde{Y}_t)$ and $(\tilde{X}_t,Y_t)$  have the same diffusion terms and the difference in drift is engineered so that Girsanov's theorem  implies that the distributions of $(X,\tilde Y)|_{[0,T]}$ and $(\tilde X,Y)|_{[0,T]}$ are absolutely continuous and leads to a formula for the relative entropy.  The coupling between $\tilde X_t$  and $Y_t$ in Eq.~\ref{eq:tilde_X} forces them to remain close pathwise, despite the difference in  their diffusion terms. The average size of the coupling term can be thought of as the cost of replacing $\eta$ with $\sigma$ and will be shown to scale with the size of  $\eta-\sigma$. The derivation will then proceed as follows.

\begin{enumerate}
\item\label{proof_step1} First we obtain bounds on the difference between observables of $\tilde X_t$ and $Y_t$  by using stochastic calculus estimates to compare these processes pathwise, assuming a sufficiently strong coupling strength $\kappa$ in \eqref{eq:tilde_X}. An effective pathwise comparison   necessitates using the same driving-noise process, $W_t$,  in the equations for both  $\tilde X_t$ and $Y_t$, as indicated in \eqref{eq:tilde_X}-\eqref{eq:Y}.
\item Next we use the information-theoretic UQ method (see the references in Section \ref{sec:related} and especially \cite{dupuis2016path,Birrell2020})  to bound the difference in expectations of observables   of $\{X_t\}_{t\in[0,T]}$ and $\{\tilde X_t\}_{t\in[0,T]}$.  The resulting bound  is determined by two contributions.

{\bf  Moment Generating Function:} The first involves the moment generating function (MGF)  of the   observable under the baseline process $X_t$.    For this we will employ standard tail-bounds (sub-Gaussian and sub-exponential bounds). The case of time-averaged observables requires additional care to ensure the bounds behave well as $T\to\infty$.  In that case we build on the strategy  from  \cite{WU2000435, cattiaux_guillin_2008, doi:10.1137/S0040585X97986667} for obtaining concentration inequalities, and which was previously used in \cite{Birrell2020} to obtain sensitivity bounds for SDEs with perturbed drift.  Here we generalize this technique to  time-dependent functions, $h(t,x)$,   by combining it with the Feynman-Kac semigroup bound from \cite{birrell2026quasistatic}.   

{\bf Relative Entropy:} The second contribution to the information-theoretic UQ bound  is  the relative entropy between the distributions of $\{(X_t,\tilde Y_t)\}_{t\in[0,T]}$ and $\{(\tilde X_t,Y_t)\}_{t\in[0,T]}$ on path space. This will be bounded using Girsanov's theorem, as was done in \cite{dupuis2016path} for SDEs with perturbed drift; here it is key that the two systems have the same diffusion terms.  We note that, as both components of each system must use the same driving noise process, the  change of measure in Girsanov's theorem necessarily changes the drift of both components when transforming from \eqref{eq:X}-\eqref{eq:tilde_Y} to \eqref{eq:tilde_X}-\eqref{eq:Y}.  The form of  Eq.~\eqref{eq:tilde_Y}  is designed to accommodate this change in drift, but  does not otherwise   directly impact any of the required bounds.  Girsanov's theorem leads to a formula for the relative entropy that has two main contributions, the first coming from the difference in drifts $F-G$ and  the second  involving the   size of the coupling term; the latter will be bounded via the estimates from stage \ref{proof_step1} of the proof.  In addition, the estimates from stage \ref{proof_step1}  are needed to prove  Novikov's condition, which justifies the use of Girsanov's theorem; the change in drift inherently involves the coupling term, which is unbounded, thus making this justification  non-trivial.   

\end{enumerate}
The above two-stage method is capable of producing sensitivity bounds that are well-behaved at large $T$, as previewed in \eqref{eq:time_avg_bounds_preview}.  We emphasize that a more naive direct comparison of $X_t$ and $Y_t$ using the type of estimates described in  stage \ref{proof_step1} above leads to bounds that grow exponentially in $T$, while the information-theoretic UQ methods are not directly applicable to \eqref{eq:(X,Y)} due to the difference in diffusion terms.   Thus our approach crucially relies on the auxiliary processes $\tilde{X}_t,\tilde{Y}_t$ and the combination of the two   techniques described above.

\section{Applications}\label{sec:applications}
  In this section we show how our results can be applied to derive sensitivity bounds for time-averaged and discounted observables as well as for linear parabolic PDEs.  We will work with SDEs satisfying  the   assumptions below.  Additional required assumptions and  proof details for the underlying estimates are given in Section \ref{sec:proofs}.
\begin{assumption}\label{assump:main_assumptions_1}
Let $W_t$ be a $d$-dimensional Wiener process  on a complete probability space $(\Omega,\mathcal{F},\mathbb{P})$ and suppose $F,G:[0,\infty)\times\mathbb{R}^d\to\mathbb{R}^d$ and $\sigma,\eta:[0,\infty)\times\mathbb{R}^d\to\mathbb{R}^{d\times d}$ satisfy the following: 
\begin{enumerate}
\item \label{assump:existence_uniqueness}  Assume $F(t,x)$, $G(t,x)$, $\sigma(t,x)$, and $\eta(t,x)$ are continuous,   globally Lipschitz  in $x$, uniformly in $t$, and that  $\eta$ and $\sigma$ are uniformly bounded and $F,G$ are linearly bounded in $x$, uniformly in $t$.
\item \label{assump:ICS}    Assume the initial conditions for the SDEs   satisfy $\tilde X_{0}=X_0$, $\tilde{Y}_0=Y_0$ and suppose there exists $\alpha,\tilde{\alpha}\geq 1$, $\beta,\tilde{\beta}>0$ such that 
\begin{align}
 \mathbb{E}[\|  \tilde{X}_0-Y_0\|^{2n}]\leq \alpha \beta^n n!\,,\,\,\,  \mathbb{E}[\|  \tilde{X}_0\|^{2n}]\leq \tilde{\alpha} \tilde{\beta}^n n!
\end{align}
for all $n\in\mathbb{Z}^+$, where we use $\|\cdot\|$ to denote the $\ell^2$ norm on $\mathbb{R}^d$. 
\begin{remark}
This holds if the initial distributions $\tilde{X}_0$, $Y_0$ are sub-Gaussian; see, e.g., Proposition 2.5.2 in \cite{vershynin2018high} (specifically, the proof of (i)$\implies$(ii)).
\end{remark}
\item \label{assump:G_decomp}  Assume we have a decomposition  $G=G_0+\tilde{G}$ where $\tilde{G}$ is globally $L_{\tilde{G}}$-Lipschitz in $x$, uniformly in $t$, and  $G_0$ satisfies a uniform monotonicity condition for some $C_G\geq 0$:
\begin{align}
(x-y)\cdot (G_0(t,x)-G_0(t,y))\leq -C_G\|x-y\|^2 \text{ for all }t\geq 0, x,y\in\mathbb{R}^d\,.
\end{align}
\begin{remark}
Letting $G_0=0$ (and hence $C_G=0$ and $L_{\tilde{G}}=L_G$, the Lipschitz constant for $G$) is covered by our results; a nonzero $C_G$ will result in tighter bounds.
\end{remark}
\item \label{assump:G_confining}  Assume there exists $A_G\geq 0$, $B_G>0$ such that
\begin{align}
x\cdot G(t,x)\leq A_G-B_G\|x\|^2 \,\, \text{  for all } (t,x)\in[0,\infty)\times\mathbb{R}^d\,.
\end{align}
\begin{remark}
This is a confining assumption on $G$; e.g., it holds if $G(t,x)=-r(t,x) x$ where $r(t,x)$ is bounded on $[0,\infty)\times B_R(0)$ and is bounded below by a positive constant on $[0,\infty)\times B_R^c(0)$ for some $R>0$.
\end{remark}
\item \label{assump:epsilon_lambda}  Let $\epsilon>0$ and choose $\kappa\geq 0$   large enough that  
\begin{align}\label{eq:K_gamma_eps}
K_{\kappa,\epsilon}\coloneqq 2(\kappa+C_G-L_{\tilde{G}}  -(1+\epsilon)L_{\sigma}^2/2)>0\,,
\end{align}
where $L_{\sigma}$ denotes   the Lipschitz constant for $\sigma$ in $x$ (assumed to hold uniformly in $t$) under the  Frobenius norm, $\|\cdot\|_F$. 
\item \label{assump:delta} Choose  $\delta>0$ small enough such that   
\begin{align}\label{eq:tilde_K_gamma_delta}
\tilde{K}_{\kappa,\delta}\coloneqq 2B_G-\delta \kappa>0\,.
\end{align}

\end{enumerate}
\end{assumption}
\begin{remark}\label{remark:existence_uniqueness}
Note that Assumptions \ref{assump:main_assumptions_1}\,.\,\ref{assump:existence_uniqueness}\,-\,\ref{assump:ICS}  together imply    global existence and uniqueness of strong solutions to  \eqref{eq:X}-\eqref{eq:tilde_Y} and   to \eqref{eq:tilde_X}-\eqref{eq:Y}. They also satisfy  $\mathbb{E}[\sup_{t\in[0,T]} \|X_t\|^p]<\infty$ for all $T>0$ and similarly for $\tilde X_t$, $Y_t$. For required background on SDE theory,  see, e.g.,  \cite{karatzas2014brownian}.
\end{remark}

\subsection{Time-Averaged Observables}

First we detail the results of our method when applied to time-averaged observables, assuming that the SDE for $X_t$ is time homogeneous and starts in an invariant distribution, $\mu_X^*$.  For definitions of the sub-Gaussian and Bernstein MGF bounds used below see \eqref{eq:sub_Gauss_def} and \eqref{eq:Bernstein_def} respectively.  For further background on such MGF bounds, including tools for obtaining them, see, e.g.,  Chapter 2 in \cite{vershynin2018high}, Chapter 2 in \cite{wainwright2019high} and Chapter 5 in \cite{bakry2013analysis}. Here and in the following we define    $\|\sigma-\eta\|_{F,\infty}\coloneqq\sup_{t\geq 0,x\in\mathbb{R}^d}\|\sigma(t,x)-\eta(t,x)\|_{F}$.  
\begin{theorem}\label{thm:bounds_time_avg}  
Under Assumptions \ref{assump:main_assumptions_1}, \ref{assump:h_assump1}, \ref{assump:main_assumptions_2}, and \ref{assump:sigma_invert} and with $D_{1,T},D_{2,T}$ as defined in \eqref{eq:D1_def}\,-\,\eqref{eq:D2_def} we have the following.

\begin{enumerate}
\item \label{case:time_avg_main_thm_subGauss}   Suppose that  for all $t\geq 0$, $h_t$ is $\sigma_{h_t}$-sub-Gaussian with respect to $\mu_{X}^*$ for some $\sigma_{h_t}\in(0,\infty)$ and $\sigma_{h_t}\in L^2([0,T],dt)$ for all $T>0$. Then for all $T>0$ we have
\begin{align}\label{eq:bounds_time_avg_thm_subGauss_general} 
&\left|\mathbb{E}\left[\frac{1}{T}\int_0^T h(t,Y_t)dt\right]- \frac{1}{T}\int_0^T E_{\mu_X^*}[h_t] dt \right|\leq\frac{ L_h(1+1/\epsilon)^{1/2}}{K_{\kappa,\epsilon}^{1/2}}(  1+M_h D_{1,T}) \|\sigma-\eta\|_{F,\infty}\\
& +\frac{2L_h(1-e^{-K_{\kappa,\epsilon}T/2})}{K_{\kappa,\epsilon}} (1+M_hD_{2,T}) \frac{\mathbb{E}[\|X_0-Y_0\|^{2}]^{1/2}}{T} \notag\\
&+\left(\frac{C_* }{ T}\int_0^T \sigma_{h_t}^2 dt \frac{1}{T} \int_0^T\mathbb{E}\left[ \left\|\sigma^{-1}(\tilde{X}_t)\left(F(\tilde{X}_t)-G(t,\tilde{X}_t)+\kappa(\tilde{X}_t-Y_t)\right)\right\|^2\right]dt\right)^{1/2}\notag\,.
\end{align}

\item \label{case:time_avg_main_thm_subexp} Suppose that for all $t\geq 0$,   $h_t$ satisfies a $(\sigma_{h_t},b)$-Bernstein MGF bound with respect to $\mu_{X}^*$  with $\sigma_{h_t}\in L^2([0,T],dt)$ for all $T>0$. Then   for all $T>0$ we have
\begin{align}\label{eq:bounds_time_avg_thm_subexp_general} 
&\left|\mathbb{E}\left[\frac{1}{T}\int_0^T h(t,Y_t)dt\right]- \frac{1}{T}\int_0^T E_{\mu_X^*}[h_t] dt \right|\leq\frac{ L_h(1+1/\epsilon)^{1/2}}{K_{\kappa,\epsilon}^{1/2}}(  1+M_h D_{1,T}) \|\sigma-\eta\|_{F,\infty} \\
&+\frac{2L_h(1-e^{-K_{\kappa,\epsilon}T/2})}{K_{\kappa,\epsilon}} (1+M_hD_{2,T}) \frac{\mathbb{E}[\|X_0-Y_0\|^{2}]^{1/2}}{T} \notag\\
&+\left(\frac{C_* }{ T}\int_0^T \sigma_{h_t}^2 dt \frac{1}{T} \int_0^T\mathbb{E}\left[ \left\|\sigma^{-1}(\tilde{X}_t)\left(F(\tilde{X}_t)-G(t,\tilde{X}_t)+\kappa(\tilde{X}_t-Y_t)\right)\right\|^2\right]dt\right)^{1/2}\notag\\
&+\frac{C_*b}{2T}\int_0^T\mathbb{E}\left[ \left\|\sigma^{-1}(\tilde{X}_t)\left(F(\tilde{X}_t)-G(t,\tilde{X}_t)+\kappa(\tilde{X}_t-Y_t)\right)\right\|^2\right]dt\notag\,.
\end{align}

\end{enumerate}

In both cases, if   $F-G$ and $\sigma_{h_t}$ are also uniformly bounded  and $X_0=Y_0$   then 
\begin{align}\label{eq:bounds_time_avg_thm_subGauss_bigO} 
\sup_{T>0}\left|\mathbb{E}\left[\frac{1}{T}\int_0^T h(t,Y_t)dt\right]- \frac{1}{T}\int_0^T E_{\mu_X^*}[h_t] dt \right|
=&O\left(\|F-G\|_\infty+\|\sigma-\eta\|_{F,\infty}\right)\,.
\end{align}
\end{theorem}
\begin{remark}\label{remark:uniform_dependence}
This proves the result \eqref{eq:time_avg_bounds_preview} which was previewed above.  We emphasize that the implicit constant in \eqref{eq:bounds_time_avg_thm_subGauss_bigO}  depends  on the norm of $\sigma^{-1}$ and so the bound does not hold uniformly as one takes the diffusion $\sigma$ to zero.     We also note that, even when $F-G$ is not uniformly bounded, explicit finite bounds on \eqref{eq:bounds_time_avg_thm_subGauss_general}  and \eqref{eq:bounds_time_avg_thm_subexp_general}  can be obtained using  an appropriate   bound on the growth of  $F-G$ along with the moment bound \eqref{eq:E_X_2_bound}. 
Similar remarks apply to the other applications we cover below.
\end{remark}
\begin{proof}
First use  Lemma \ref{lemma:stage1_result_time_avg}  to bound the difference between observables of $Y_t$ and $\tilde{X}_t$:
\begin{align}\label{eq:time_avg_proof_X_tilde_Y_step}
&\left|\mathbb{E}\left[\frac{1}{T}\int_0^T h(t,\tilde X_t)dt\right]-\mathbb{E}\left[\frac{1}{T}\int_0^T h(t,Y_t)dt\right]\right| 
\leq\frac{ L_h(1+1/\epsilon)^{1/2}}{K_{\kappa,\epsilon}^{1/2}}(  1+M_h D_{1,T}) \|\sigma-\eta\|_{F,\infty} \\
&+\frac{2L_h(1-e^{-K_{\kappa,\epsilon}T/2})}{K_{\kappa,\epsilon}} (1+M_hD_{2,T}) \frac{\mathbb{E}[\|X_0-Y_0\|^{2}]^{1/2}}{T} \notag
\end{align}
 for all $T>0$. Next apply the UQ bound from Lemma \ref{lemma:UQ_bound_general} to the observable on path space, $g:C([0,T],\mathbb{R}^d)\to\mathbb{R}$, $g(\gamma)=\int_0^T h(t,\gamma_t)dt$ and divide by $T$ to obtain a bound on the  difference between observables of $X_t$ and $\tilde{X}_t$:
\begin{align}\label{eq:UQ_bound_X_tilde_X_time_avg}
&\pm\left(\mathbb{E}\left[\frac{1}{T}\int_0^T h(t,\tilde{X}_t)dt\right]- \frac{1}{T}\int_0^T E_{\mu_X^*}[h_t] dt \right)\\
\leq& \inf_{c>0}\left\{\frac{1}{cT}\log \mathbb{E}\left[\exp\left(\pm c\int_0^T \hat{h}(t,X_t)dt \right)\right]+\frac{1}{cT}\mathrm{KL}\left(P_{\tilde{X}|_{[0,T]}}\|P_{{X}|_{[0,T]}}\right)\right\}\,,\notag
\end{align}
where $\hat{h}(t,x)\coloneqq h(t,x)-E_{\mu_X^*}[h_t]$. Moreover,   Lemma \ref{lemma:KL_bound} provides the KL-divergence bound
\begin{align}\label{eq:time_avg_proof_KL_step}
\mathrm{KL}\left(P_{\tilde{X}|_{[0,T]}}\|P_{{X}|_{[0,T]}}\right)\leq& \frac{1}{2}\int_0^T\mathbb{E}\left[ \left\|\sigma^{-1}(\tilde{X}_t)\left(F(\tilde{X}_t)-G(t,\tilde{X}_t)+\kappa(\tilde{X}_t-Y_t)\right)\right\|^2\right]dt\,.
\end{align}

Now   consider the two tail-behavior cases.
\begin{enumerate}
\item In the sub-Gaussian case \ref{case:time_avg_main_thm_subGauss}, we combine \eqref{eq:UQ_bound_X_tilde_X_time_avg} with Lemma \ref{lemma:MGF_bound_time_avg_subGauss}   and evaluate the infimum over $c$ to obtain
\begin{align}
&\left|\mathbb{E}\left[\frac{1}{T}\int_0^T h(t,\tilde{X}_t)dt\right]- \frac{1}{T}\int_0^T E_{\mu_X^*}[h_t] dt \right|\\
\leq& \inf_{c>0}\left\{\frac{C_* c}{2 T}\int_0^T \sigma_{h_t}^2 dt+\frac{1}{cT}\mathrm{KL}\left(P_{\tilde{X}|_{[0,T]}}\|P_{{X}|_{[0,T]}}\right)\right\}\notag\\
=&\sqrt{\frac{2C_* }{ T}\int_0^T \sigma_{h_t}^2 dt \frac{1}{T}\mathrm{KL}\left(P_{\tilde{X}|_{[0,T]}}\|P_{{X}|_{[0,T]}}\right)}\,.\notag
\end{align}
 for all $T>0$. Combining this with \eqref{eq:time_avg_proof_X_tilde_Y_step} and \eqref{eq:time_avg_proof_KL_step} and using the triangle inequality we obtain the claimed result.

If $F-G$ and  $\sigma_{h_t}$ are uniformly bounded and $X_0=Y_0$ then, using  \eqref{eq:bounds_time_avg_thm_subGauss_general}, the bound on $\mathbb{E}[\|\tilde{X}_t-Y_t\|^2]$ from  \eqref{eq:E_Z_2_bound_uniform}, and the facts that $\sup_{T>0}(1-e^{-K_{\kappa,\epsilon}T/2})/T<\infty$ and $\sup_{T>0}D_{i,T}<\infty$, we can further compute
\begin{align}
&\sup_{T>0}\left|\mathbb{E}\left[\frac{1}{T}\int_0^T h(t,Y_t)dt\right]- \frac{1}{T}\int_0^T E_{\mu_X^*}[h_t] dt \right|\\
\leq&\frac{ L_h(1+1/\epsilon)^{1/2}}{K_{\kappa,\epsilon}^{1/2}}(  1+M_h D_{1,T}) \|\sigma-\eta\|_{F,\infty}\notag\\
&+\left(  2C_* (\sup_t  \sigma_{h_t})^2\|\sigma^{-1}\|_{2,\infty}^2\left(\|F-G\|_\infty^2+\frac{\kappa^2(1+1/\epsilon)\|\sigma-\eta\|_{F,\infty}^2}{K_{\kappa,\epsilon}}\right)\right)^{1/2}\notag\\
=&O\left(\|F-G\|_\infty+\|\sigma-\eta\|_{F,\infty}\right)\notag\,.
\end{align}

\item The proof in the Bernstein case   follows from Lemma \ref{lemma:MGF_bound_time_avg_subexp} by  similar calculations to the previous case; see Appendix \ref{app:time_avg_sub_exp} for details.
\end{enumerate}
\end{proof}
We emphasize that our method produces a   uniform-in-time bound \eqref{eq:bounds_time_avg_thm_subGauss_bigO}.
A more elementary argument based on directly  comparing the SDEs for $X_t$ and $Y_t$ via Gronwall's inequality would produce bounds that grow exponentially in time and would become uninformative as $T\to\infty$; similar comments apply to the other applications in this work.

\subsubsection{Wasserstein Distance Between Invariant Distributions}
When the SDE for $Y_t$ is also time-homogeneous and is started in an invariant distribution, by optimizing over all $1$-Lipschitz functions in Theorem \ref{thm:bounds_time_avg}, we obtain the following bound on the $1$-Wasserstein distance between the invariant distributions in terms of the difference between the   diffusions and drifts.  In particular, this result implies \eqref{eq:1_Wasserstein_bounds_preview}.
\begin{theorem}\label{thm:Wasserstein}
Under Assumptions  \ref{assump:main_assumptions_1}, \ref{assump:h_assump1}, \ref{assump:main_assumptions_2},  and \ref{assump:sigma_invert}, suppose also that:
\begin{enumerate}
\item  The SDE \eqref{eq:Y} for $Y_t$ is time-homogeneous.
\item $Y_0$ is distributed as $\mu^*_Y$, which is an invariant distribution for  \eqref{eq:Y}.
\item There exists $\sigma_{\mathrm{Lip}_1}\in (0,\infty)$ such that every $1$-Lipschitz function $h:\mathbb{R}^d\to\mathbb{R}$  is $\sigma_{\mathrm{Lip}_1}$-sub-Gaussian   with respect to $\mu_{X}^*$.
\end{enumerate}
If $F-G$ is bounded then we have the following bound on the $1$-Wasserstein distance:
\begin{align}\label{eq:W1_BigO_bound}
W_1(\mu_Y^*,\mu_{X}^*)\leq& \frac{  (1+1/\epsilon)^{1/2}}{K_{\kappa,\epsilon}^{1/2}} \|\sigma-\eta\|_{F,\infty}\\
&+(2C_*)^{1/2}  \sigma_{\mathrm{Lip}_1}\|\sigma^{-1}\|_{2,\infty} 
\left(  \|F-G\|_\infty^2+\frac{\kappa^2(1+1/\epsilon)\|\sigma-\eta\|_{F,\infty}^2}{K_{\kappa,\epsilon}}\right)^{1/2}\notag\,.
\end{align}
\begin{comment} Without boundedness of $F-G$ one still has
\begin{align}\label{eq:W1_bound_general}
W_1(\mu_Y^*,\mu_{X}^*)\leq&\frac{  (1+1/\epsilon)^{1/2}}{K_{\kappa,\epsilon}^{1/2}} \|\sigma-\eta\|_{F,\infty} \\
&+C_*^{1/2}  \sigma_{\mathrm{Lip}_1}\liminf_{T\to\infty}
\left(  \frac{1}{T} \int_0^T\mathbb{E}\left[ \left\|\sigma^{-1}(\tilde{X}_t)\left(F(\tilde{X}_t)-G(\tilde{X}_t)+\kappa(\tilde{X}_t-Y_t)\right)\right\|^2\right]dt\right)^{1/2}\,.\notag
\end{align}
\end{comment}
\end{theorem}
\begin{remark}
That every Lipschitz function is sub-Gaussian under $\mu_X^*$ can be proven using   the   $\log$-Sobolev inequality from Assumption \ref{assump:main_assumptions_2} by  combining the form of the carr{\'e} du champ operator for a  uniformly elliptic SDE together with, e.g.,  Theorem 5.4.1 in \cite{bakry2013analysis}.  In   the case where $F$ is a gradient, $F=-\nabla V$, it can can also be obtained  by more elementary means using the well-known  explicit formula for the invariant distribution together with an appropriate growth rate of growth  of $V$   at infinity.
\end{remark}
\begin{proof}
Let $h:\mathbb{R}^d\to\mathbb{R}$ be a $1$-Lipschitz function.  Then $h$ satisfies Assumption \ref{assump:h_assump1} with $M_h=0$, $L_h\leq 1$.  Applying case \ref{case:time_avg_main_thm_subGauss} of Theorem \ref{thm:bounds_time_avg}  therefore gives
\begin{align}
&\left| \int h d\mu_Y^*-\int   h d\mu_{X}^*\right|\\
\leq&\frac{  (1+1/\epsilon)^{1/2}}{K_{\kappa,\epsilon}^{1/2}} \|\sigma-\eta\|_{F,\infty} +\frac{2(1-e^{-K_{\kappa,\epsilon}T/2})}{K_{\kappa,\epsilon}}   \frac{\mathbb{E}[\|X_0-Y_0\|^{2}]^{1/2}}{T} \notag\\
&+\left(C_*  \sigma_{\mathrm{Lip}_1}^2  \frac{1}{T} \int_0^T\mathbb{E}\left[ \left\|\sigma^{-1}(\tilde{X}_t)\left(F(\tilde{X}_t)-G(\tilde{X}_t)+\kappa(\tilde{X}_t-Y_t)\right)\right\|^2\right]dt\right)^{1/2}\notag
\end{align}
for all $T>0$.  The right-hand side no longer depends on $h$, therefore by maximizing over $1$-Lipschitz functions, using the Kantorovich-Rubinstein formula for $W_1$, the bound \eqref{eq:E_Z_2_bound_uniform}, and taking $T\to\infty$ we  obtain \eqref{eq:W1_BigO_bound}.
\end{proof}

By considering the special case of OU processes, we can show that the linear scaling of  \eqref{eq:W1_BigO_bound} in   $\|F-G\|_\infty$   and in  $\|\sigma-\eta\|_{F,\infty}$ are both optimal.  Consider the  linear SDEs on $\mathbb{R}$,
\begin{align}
dX_t=&-\frac{1}{2}(X_t-\mu_X) dt+\sigma dW_t\,,\,\,\,\,\,\,\,
dY_t=-\frac{1}{2}(Y_t-\mu_Y) dt+\eta dW_t\,,
\end{align}
whose respective invariant distributions are the normal distributions  $\mu_{X}^*=N(\mu_X,\sigma^2)$, $\mu_{Y}^*=N(\mu_Y,\eta^2)$.  Now consider the following two cases.

\begin{enumerate}
\item $\sigma=\eta$:  In this case, using the Kantorovich-Rubinstein formula for $W_1$, we   lower bound the $1$-Wasserstein distance by the difference in expected values of the $1$-Lipschitz function $h(x)=x$, which gives
\begin{align}
W_1(\mu_Y^*,\mu_X^*)\geq  |E_{N(\mu_Y,\sigma^2)}[h]-E_{N(\mu_X,\sigma^2)}[h]|=|\mu_Y-\mu_X|=2\|F-G\|_\infty\,.
\end{align}
Hence the  $O(\|F-G\|_\infty)$ bound in \eqref{eq:W1_BigO_bound} is the optimal rate.

\item $\mu\coloneqq\mu_X=\mu_Y$:   In light of Remark \ref{remark:uniform_dependence}, we should not expect the bounds to hold uniformly in $\sigma$, hence we consider a fixed value of $\sigma$ and will derive a lower bound on the $1$-Wasserstein distance that scales with $|\eta-\sigma|$ as $\eta\to \sigma$. First recall the   formula $ W_1(Q,P)=\int_{-\infty}^\infty|F_Q(t)-F_P(t)|dt$ for the $1$-dimensional $1$-Wasserstein metric  \cite{doi:10.1137/1118101},
where $F_Q$ and $F_P$ denote the cumulative distribution functions of $Q$ and $P$ respectively.  Using this along with Taylor's theorem with integral remainder we can compute
\begin{align}
&W_1(\mu_{Y}^*,\mu_X^*)=\frac{1}{2}\int_{-\infty}^\infty\left|\mathrm{erfc}\left(\frac{t-\mu}{\sqrt{2}\sigma}\right)-\mathrm{erfc}\left(\frac{t-\mu}{\sqrt{2}\eta}\right)\right|dt\\
\geq &  \left(\frac{2}{\pi}\right)^{1/2}  \frac{\eta}{\sigma}|\sigma-\eta| \notag\\
&
-   \frac{1}{(2\pi)^{1/2}} (\sigma^{-1}-\eta^{-1})^2\left( \eta^{-1}+\sigma^{-1}\right)\int_{-\infty}^\infty  \exp\left(- \left(\frac{1}{4\eta^2}+\frac{1}{2}\left(\frac{1}{\sigma}-\frac{1}{\eta}\right)^2\right)u^2\right)
  |u|^3 
du\notag\\
=&\left(\frac{2}{\pi}\right)^{1/2}   |\sigma-\eta| +O( |\sigma-\eta| ^2)\notag
\end{align}
%use-(a+b)^2<=-\frac{1}{2}a^2+b^2
as $\eta\to \sigma$. Hence the  $O(\|\sigma-\eta\|_{F,\infty})$ bound in \eqref{eq:W1_BigO_bound} is the optimal rate.
\end{enumerate}

\subsection{Discounted Observables}
Next we detail the result of our method when applied to discounted observables.  We emphasize that in this case we do not make any assumption regarding invariant initial distributions or even time-homogeneity nor do we require the use of functional inequalities; these  are not necessary when  utilizing a weight measure $\rho(dt)$ whose tail decays sufficiently fast as $t\to\infty$. 
\begin{theorem}\label{thm:bounds_discounted_obs} 
Under Assumptions   \ref{assump:main_assumptions_1}, \ref{assump:h_assump1}, and \ref{assump:sigma_invert}, assume also that $\rho(dt)$ is a finite positive measure on $[0,\infty)$ and $X_0=Y_0$. With   $\tilde{D}$ as defined in \eqref{eq:tilde_D_def}, we have the following.

\begin{enumerate}
\item  If, for all $t\geq 0$, $h_t$ is $\sigma_{h_t}$-sub-Gaussian with respect to $P_{X_t}$ for some $\sigma_{h_t}\in(0,\infty)$ then
\begin{align}
&\sup_{T\geq 0}\left|\mathbb{E}\left[\int_0^T h(t, Y_t)\rho(dt)\right]-\mathbb{E}\left[ \int_0^T h(t,X_t)\rho(dt)\right]\right|\\
&\leq\frac{L_h(1+1/\epsilon)^{1/2}\|\sigma-\eta\|_{F,\infty}}{K_{\kappa,\epsilon}^{1/2}}
\left(\|\rho\|_{\mathrm{TV}}+M_h\tilde{D}\right)\notag\\
&+ \int_0^\infty  \sigma_{h_t}      \left({ \int_0^t\mathbb{E}\left[ \left\|\sigma^{-1}(s,\tilde{X}_s)\left(F(s,\tilde{X}_s)-G(s,\tilde{X}_s)+\kappa(\tilde{X}_s-Y_s)\right)\right\|^2\right]ds} \right)^{1/2}\rho(dt)\notag\,.
\end{align}

If, in addition, $F-G$ is   bounded and $t^{1/2} \sigma_{h_t}  \in L^1([0,\infty),\rho(dt))$  then we  obtain 
\begin{align}\label{eq:discounted_BigO_bound_general_weight_sub_gauss}
&\sup_{T\geq 0}\left|\mathbb{E}\left[\int_0^T h(t, Y_t)\rho(dt)\right]-\mathbb{E}\left[ \int_0^T h(t,X_t)\rho(dt)\right]\right|=O(\|F-G\|_\infty+\|\sigma-\eta\|_{F,\infty})\,.
\end{align} 
\item   If, for all $t\geq 0$, $h_t$ satisfies a $(\sigma_{h_t},b)$-Bernstein MGF bound with respect to $P_{X_t}$  for some $\sigma_{h_t}\in(0,\infty)$ then
\begin{align}
&\sup_{T\geq 0}\left|\mathbb{E}\left[\int_0^T h(t, Y_t)\rho(dt)\right]-\mathbb{E}\left[ \int_0^T h(t,X_t)\rho(dt)\right]\right| \\
\leq&\frac{L_h(1+1/\epsilon)^{1/2}\|\sigma-\eta\|_{F,\infty}}{K_{\kappa,\epsilon}^{1/2}}
\left(\|\rho\|_{\mathrm{TV}}+M_h\tilde{D}\right)\notag\\
&+\int_0^\infty\sigma_{h_t} \left({ \int_0^t\mathbb{E}\left[ \left\|\sigma^{-1}(s,\tilde{X}_s)\left(F(s,\tilde{X}_s)-G(s,\tilde{X}_s)+\kappa(\tilde{X}_s-Y_s)\right)\right\|^2\right]ds}\right)^{1/2}\rho(dt)\notag\\
&+\frac{b}{2}\int_0^\infty \int_0^t\mathbb{E}\left[ \left\|\sigma^{-1}(s,\tilde{X}_s)\left(F(s,\tilde{X}_s)-G(s,\tilde{X}_s)+\kappa(\tilde{X}_s-Y_s)\right)\right\|^2\right]ds\,\rho(dt)\notag\,.
\end{align}

If, in addition, $F-G$ is   bounded and $t^{1/2}\sigma_{h_t} ,t  \in L^1([0,\infty),\rho(dt))$  then we  obtain 
 \begin{align}
&\sup_{T\geq 0}\left|\mathbb{E}\left[\int_0^T h(t, Y_t)\rho(dt)\right]-\mathbb{E}\left[ \int_0^T h(t,X_t)\rho(dt)\right]\right| =O(\|F-G\|_\infty+\|\sigma-\eta\|_{F,\infty})\,.
\end{align}

\end{enumerate}

\end{theorem}
\begin{remark}
In particular, under appropriate assumptions, applying this result to the exponentially discounted weight $\rho(dt)=e^{-rt}dt$, $r>0$,  yields  \eqref{eq:discounted_cost_bounds_preview}.
\end{remark}
\begin{proof}
First use Lemma \ref{lemma:stage1_result_discount}   to compute
\begin{align}\label{eq:discounted_observable_X_tilde_Y_temp}
&\sup_{T\geq 0}\left|\mathbb{E}\left[\int_0^T h(t,\tilde X_t)\rho(dt)\right]-\mathbb{E}\left[ \int_0^T h(t,Y_t)\rho(dt)\right]\right| \\
\leq&\frac{L_h(1+1/\epsilon)^{1/2}\|\sigma-\eta\|_{F,\infty}}{K_{\kappa,\epsilon}^{1/2}}
\left(\|\rho\|_{\mathrm{TV}}+M_h\tilde{D}\right)\notag\,,
\end{align}
where $\|\cdot\|_{\mathrm{TV}}$ denotes the total variation.
Next apply   Lemma \ref{lemma:UQ_bound_general} to  the observable $g(\gamma)=h(t,\gamma_t)$ for each $t$ and then integrate with respect to $\rho(dt)$ and maximize over $T\geq0$ to obtain  
\begin{align}\label{eq:UQ_bound_X_tilde_X_discounted}
&\sup_{T\geq 0}\left\{\pm \left(\int_0^T\mathbb{E}[h(t,\tilde{X}_t)] \rho(dt)-\int_0^T\mathbb{E}[h(t,X_t)] \rho(dt)\right) \right\}\\
\leq& \int_0^\infty \inf_{c>0}\left\{\frac{1}{c}\log E_{{X_t}}\left[\exp(\pm c(h_t-E_{{X_t}}[h_t]))\right]+\frac{1}{c}\mathrm{KL}\left(P_{\tilde{X}|_{[0,t]}}\|P_{{X}|_{[0,t]}}\right)\right\} \rho(dt)\,,\notag
\end{align}
where $E_{X_t}$ denotes the expectation with respect to the distribution of ${X_t}$;   Lemma \ref {lemma:KL_bound}
implies the KL-divergence bound
\begin{align}\label{eq:KL_bound_discounted_proof}
\mathrm{KL}\left(P_{\tilde{X}|_{[0,t]}}\|P_{{X}|_{[0,t]}}\right)\leq& \frac{1}{2}\int_0^t\mathbb{E}\left[ \left\|\sigma^{-1}(s,\tilde{X}_s)\left(F(s,\tilde{X}_s)-G(s,\tilde{X}_s)+\kappa(\tilde{X}_s-Y_s)\right)\right\|^2\right]ds
\end{align}
 for all $t>0$.   In both the sub-Gaussian and Bernstein cases, the remainder of the proof closely follows the corresponding steps in the proof of Theorem \ref{thm:bounds_time_avg}, hence we omit the details.
\end{proof}

\subsection{Sensitivity Bounds for Linear Parabolic PDEs}

As our final application, we derive perturbation bounds for classical solutions to linear parabolic PDEs of the form \eqref{eq:Feynman_Kac_pde_intro}
 by applying our method to the Feynman-Kac formula; in particular, the following result implies \eqref{eq:PDE_bounds_preview}. We focus on perturbations to the drift, $F$, and especially the diffusion, $\sigma$, as perturbations to the other terms in the PDE  are relatively trivial to handle since they do not change the underlying SDE. 
\begin{theorem}\label{thm:parabolic_PDE}
In addition to Assumptions \ref{assump:main_assumptions_1}, \ref{assump:PDE}, \ref{assump:main_assumptions_2}, and \ref{assump:sigma_invert},  assume the following:
\begin{enumerate}
\item Let $T>0$, $u_T(t,x)$ be a $C^{1,2}$ solution to \eqref{eq:Feynman_Kac_pde_intro}, and $v_T(t,x)$ be a $C^{1,2}$ solution to \eqref{eq:Feynman_Kac_pde_intro} but with $F$ replaced by $G$ and $\sigma$ replaced by $\eta$. 
\item   Suppose $u_T(t,x)$ and $v_T(t,x)$ are both polynomially-bounded in $x$, uniformly in $t\in[0,T]$.
\item   Suppose $f$ is $\sigma_f$-sub-Gaussian with respect to $\mu_X^*$.
\item Suppose that for all $t\in[0,T]$, $h_t$ is $\sigma_{h_t}$-sub-Gaussian with respect to $\mu_{X}^*$ for some $\sigma_{h_t}\in(0,\infty)$ and $\sigma_{h_t}\in L^2([0,T],dt)$.
\end{enumerate}

Then for all $t\in[0,T]$  we have
 \begin{align}\label{eq:PDE_result}
&\|v_T(t,\cdot)-u_T(t,\cdot)\|_{L^1(\mu_X^*)}\\
\leq& (  (T-t) L_h + L_f)\frac{(1+1/\epsilon)^{1/2}\|\sigma-\eta\|_{F,\infty}}{K_{\kappa,\epsilon}^{1/2}}+2 \left(\sigma_f^2   + C_* \int_0^{T-t}\sigma^2_{h_{t+r}} dr\right)^{1/2}\notag \\
&\qquad\times\left(\int_t^{T} \int\mathbb{E}\left[ \left\|\sigma^{-1}(\tilde{X}^{t,x}_{s})\left(F(\tilde{X}^{t,x}_{s})-G(s,\tilde{X}^{t,x}_{s})+\kappa(\tilde{X}^{t,x}_{s}-Y^{t,x}_{s})\right)\right\|^2\right]\mu_X^*(dx) \,ds\right)^{1/2}\notag.
\end{align}

If, in addition, $F-G$ and $\sigma_{h_t}$ are bounded then
\begin{align}\label{eq:PDE_result_bounded_case}
&\|v_T(t,\cdot)-u_T(t,\cdot)\|_{L^1(\mu_X^*)}\leq (  (T-t) L_h + L_f)\frac{(1+1/\epsilon)^{1/2}\|\sigma-\eta\|_{F,\infty}}{K_{\kappa,\epsilon}^{1/2}}\\
&+2\sqrt{2}  \|\sigma^{-1}\|_{2,\infty} (T-t)^{1/2} \left(C_*(T-t) (\sup_s\sigma_{h_{s}})^2+\sigma_f^2   \right)^{1/2}\notag \\
&\qquad\qquad\qquad\qquad\times\left(\|F-G\|_\infty^2+\kappa^2    \frac{(1+1/\epsilon)\|\sigma-\eta\|_{F,\infty}^2}{K_{\kappa,\epsilon}}  \right)^{1/2}\notag\\
=&O(\|F-G\|_\infty+\|\sigma-\eta\|_{F,\infty})\,.\notag
\end{align}

\end{theorem}
\begin{remark}
When $h\neq 0$, the scaling of the bound \eqref{eq:PDE_result_bounded_case} with $T-t$ cannot be improved in general, as the error due to the integral terms in \eqref{eq:Feynman_Kac_X} will often scale with $T-t$.  However, when $L_h=0$ there is a spurious $(T-t)^{1/2}$ dependence in our result, which arises from the linear scaling of the relative entropy bound in $T-t$. Whether this can be addressed by an improved argument is a question we leave for future work.  We also note that the argument below can be modified to produce pointwise bounds; however, it is the integration with respect to $\mu_X^*$ and the accompanying use of  Lemma \ref{lemma:MGF_bound_time_avg_subGauss} which leads to the optimal $O(T-t)$ scaling.
\end{remark}
\begin{proof}
For $x\in\mathbb{R}^d$, $t\in[0,T]$, let $\{X^{t,x}_s\}_{s\geq t}$ be the solution to \eqref{eq:X} started at position $x$ at time $t$ and let $\{(\tilde{X}^{t,x}_s,Y^{t,x}_s)\}_{s\geq t}$ be the solution to  \eqref{eq:tilde_X}\,-\,\eqref{eq:Y} started at position $(x,x)$ at time $t\geq 0$.  The Feynman-Kac formula implies the PDE solutions $u_T$ and $v_T$ can be written in terms of the SDE solutions as follows:
\begin{align}
{u}_T(t,x)\coloneqq&\mathbb{E}\left[f({X}^{t,x}_T) +\int_t^T h(s,{X}_s^{t,x})  ds\right]\,, \,\,\,\,\,\,v_T(t,x)\coloneqq\mathbb{E}\left[f(Y^{t,x}_T)  +\int_t^T h(s,Y_s^{t,x})   ds\right]\,. \label{eq:Feynman_Kac_X}
\end{align}
  See, e.g.,  Theorem 7.6 in Chapter 5 of  \cite{karatzas2014brownian} for a proof of these representations.

Now define the analogous quantity in terms of $\tilde{X}^{t,x}$,
\begin{align}
\tilde{u}_T(t,x)\coloneqq&\mathbb{E}\left[f(\tilde{X}^{t,x}_T) +\int_t^T h(s,\tilde{X}_s^{t,x})  ds\right]\,,
\end{align}
and define the observable $g:C([t,T],\mathbb{R}^d)\to\mathbb{R}$ by $g(\gamma)\coloneqq f(\gamma_T)  +\int_t^T h(s,\gamma_s)  ds$. 
Lemma \ref{lemma:UQ_bound_general} applied to $g$ gives
\begin{align}\label{eq:UQ_bound_PDE_u_tilde_u}
&\pm\left(\tilde{u}_T(t,x)-u_T(t,x)\right)\\
\leq &\inf_{c>0}\left\{\frac{1}{c}\log E_{{{X}^{t,x}|_{[t,T]}}}\left[\exp(\pm c(g-E_{{{X}^{t,x}|_{[t,T]}}}[g]))\right]+\frac{1}{c}\mathrm{KL}(P_{\tilde{X}^{t,x}|_{[t,T]}}\|P_{{X}^{t,x}|_{[t,T]}})\right\}\,.\notag
\end{align}
Using non-negativity of the right-hand side we can compute
\begin{align}\label{eq:tilde_u_u_L1_bound1}
&\int |\tilde{u}_T(t,x)-u_T(t,x)|\mu_X^*(dx)\\
\leq &\inf_{c>0}\left\{ \frac{2}{c}\max_{\ell\in\{\pm 1\}} \int \! \log E_{{{X}^{t,x}|_{[t,T]}}}\!\!\left[\exp(  \ell c(g-E_{{{X}^{t,x}|_{[t,T]}}}[g]))\right]\mu_X^*(dx) \right.\notag\\
&\left.\qquad\qquad+\frac{1}{c}\int \mathrm{KL}(P_{\tilde{X}^{t,x}|_{[t,T]}}\|P_{{X}^{t,x}|_{[t,T]}})\mu_X^*(dx)\right\}.\notag
\end{align}

For $c>0$ and $\ell\in\{\pm 1\}$, we bound the MGF term using the Cauchy-Schwarz inequality, the   Markov property,  Jensen's inequality, and Assumption \ref{assump:main_assumptions_2}\,.\,\ref{assump:inv_dist}, which yields
\begin{align}
  &  \frac{2}{c}\int \log E_{{{X}^{t,x}|_{[t,T]}}}\!\!\left[\exp(  \ell c(g-E_{{{X}^{t,x}|_{[t,T]}}}[g]))\right]\mu_X^*(dx)  \\
\leq  &   \frac{1}{c}\int \log E_{{{X}^{t,x}_T}}\left[\exp( 2 \ell c(f-E_{{{X}^{t,x}_T}}[f]))\right]\mu_X^*(dx)\notag\\
& + \frac{1}{c}\int \log \mathbb{E}\left[\exp\left( 2 \ell c\int_t^T(h(s,X^{t,x}_s) -{E}_{X^{t,x}_s}[h_s])ds\right)\right]\mu_X^*(dx) \notag
\end{align}
\begin{align}
\leq& \frac{1}{c}\log \int  E_{{{X}^{t,x}_T}}\left[\exp( 2 \ell cf)\right]\mu_X^*(dx)-2 \ell     E_{\mu_X^*}[f]\notag\\
& + \frac{1}{c}\log \int  \mathbb{E}\left[\exp\left( 2 \ell c\int_t^Th(s,X^{t,x}_s) ds\right)\right]\mu_X^*(dx) -2 \ell    \int_t^T{E}_{\mu_X^*}[h_s]ds \notag\\
=& \frac{1}{c}\log \int \exp\left( 2 \ell c(f -   E_{\mu_X^*}[f])\right)\mu_X^*(dx)  \notag\\
&+ \frac{1}{c}\log    \mathbb{E}\left[\exp\left( 2 \ell c\int_0^{T-t}\left(h_{t+r}(X_{r}) -{E}_{\mu_X^*}[h_{t+r}]\right)dr\right)\right]    \notag\,.
\end{align}
Now use the sub-Gaussianity of $f$ along with Lemma \ref{lemma:MGF_bound_time_avg_subGauss} (note that $(r,x)\mapsto h(t+r,x)$ satisfies Assumption \ref{assump:h_assump1}) to obtain
\begin{align}
  &  \frac{2}{c}\int \log E_{{{X}^{t,x}|_{[t,T]}}}\left[\exp(  \ell c(g-E_{{{X}^{t,x}|_{[t,T]}}}[g]))\right]\mu_X^*(dx)  
\leq 2   c\left(\sigma_f^2   + C_* \int_0^{T-t}\sigma^2_{h_{t+r}} dr\right)\,.
\end{align}
Using this to bound \eqref{eq:tilde_u_u_L1_bound1} and  then evaluating the infimum over $c$ we find
\begin{align}
&\int |\tilde{u}_T(t,x)-u_T(t,x)|\mu_X^*(dx)\\
\leq & 2\sqrt{2} \left(\sigma_f^2   + C_* \int_0^{T-t}\sigma^2_{h_{t+r}} dr\right)^{1/2} \left(\int \mathrm{KL}(P_{\tilde{X}^{t,x}|_{[t,T]}}\|P_{{X}^{t,x}|_{[t,T]}})\mu_X^*(dx)\right)^{1/2}\notag\,.
\end{align}
Combining this with  Lemmas \ref{lemma:Feynman-Kac_X_tilde_Y_bound} and  \ref{lemma:KL_bound}  and using the triangle inequality we obtain  the claimed result \eqref{eq:PDE_result}. If, in addition, $F-G$ and $\sigma_{h_t}$ are bounded then we can use \eqref{eq:E_Z_2_bound_uniform} to further compute 
\eqref{eq:PDE_result_bounded_case}.
\end{proof}

We note that a large portion of the above argument can also be applied to 
\begin{align}
{u}(t,x)\coloneqq&\mathbb{E}\left[f({X}^{t,x}_T) \exp\left(\int_t^T k(s,{X}_s^{t,x}) ds\right)+\int_t^T h(s,{X}_s^{t,x}) \exp\left(\int_t^s k(r,{X}^{t,x}_r)dr\right) ds\right]\,, 
\end{align}
which, under appropriate assumptions, solves the PDE
\begin{align}
\partial_t u(t,x)=-A[u_t](t,x)-k(t,x)u(t,x)-h(t,x)\,,\,\,\,u(T,x)=f(x)\,.
\end{align}
However, effectively handling the corresponding MGF so that the bounds scale optimally in $T-t$ requires techniques beyond those discussed in Section \ref{sec:functional_ineq}.

\section{Proofs of Key Estimates}\label{sec:proofs}
In this section we provide detailed proofs of the bounds that underpin the applications in Section \ref{sec:applications} and which were outlined in Section \ref{sec:proof_outline}.

\subsection{Pathwise Comparison of $\tilde X_t$ and $Y_t$}\label{sec:Z_2n_bound} 
 In the first stage of the proof, we show that the coupling term between $\tilde{X}_t$ and $Y_t$ is sufficient to keep them close to one another, with error that depends on the difference in diffusion.  Working under Assumption \ref{assump:main_assumptions_1},  for $K\in\mathbb{R}$, $n\in\mathbb{Z}^+$, an application of It{\^o}'s formula together with Assumption \ref{assump:main_assumptions_1}\,.\,\ref{assump:G_decomp}  applied to $Z_t\coloneqq\tilde X_t-Y_t$ yields
\begin{align}\label{eq:Z_2n_Ito}
e^{Kt}\|Z_t\|^{2n}=&\|Z_0\|^{2n}+\int_0^t Ke^{Ks}\|Z_s\|^{2n}ds\\
&+\int_0^t2n e^{Ks} \|Z_s\|^{2(n-1)} (Z_s\cdot(G(s,\tilde X_s)-G(s,Y_s)) -\kappa \|Z_s\|^2)ds\notag\\
&+\int_0^t 2ne^{Ks} \|Z_s\|^{2(n-1)}Z_s\cdot (\sigma(s,\tilde X_s)-\eta(s,Y_s) )dW_s\notag\\
&+\int_0^t  2n(n-1)e^{Ks}\|Z_s\|^{2(n-2)} \|(\sigma(s,\tilde X_s)-\eta(s,Y_s))^T Z_s\|^2ds\notag\\
&+\int_0^t ne^{Ks}\|Z_s\|^{2(n-1)}\|\sigma(s,\tilde X_s)-\eta(s,Y_s)\|_F^2ds\notag\\
\leq&\|Z_0\|^{2n}- \left( 2n(\kappa+C_G-L_{\tilde{G}}  )-K\right)\int_0^t e^{Ks} \|Z_s\|^{2n}ds\notag\\
&+\int_0^t 2ne^{Ks} \|Z_s\|^{2(n-1)}Z_s\cdot (\sigma(s,\tilde X_s)-\eta(s,Y_s) )dW_s\notag\\
&+\int_0^t  2n(n-1)e^{Ks}\|Z_s\|^{2(n-2)} \|(\sigma(s,\tilde X_s)-\eta(s,Y_s))^T Z_s\|^2ds\notag\\
&+\int_0^t ne^{Ks}\|Z_s\|^{2(n-1)}\|\sigma(s,\tilde X_s)-\eta(s,Y_s)\|_F^2ds\notag\,.
\end{align}
 Based on   Remark \ref{remark:existence_uniqueness}, we have $    \mathbb{E}\left[\int_0^t (2ne^{Ks} \|Z_s\|^{2(n-1)} \|(\sigma(s,\tilde X_s)-\eta(s,Y_s) )^TZ_s\|)^2 ds\right]<\infty $
for all $t$ and hence  $\int_0^t 2ne^{Ks} \|Z_s\|^{2(n-1)}Z_s\cdot (\sigma(s,\tilde X_s)-\eta(s,Y_s) )dW_s$ is a martingale; see, e.g., Chapter 3 in  \cite{karatzas2014brownian}.  In particular, it has expected value  zero.  

In the $n=1$ case, taking the expectation and using the martingale property, we   obtain
\begin{align}
e^{Kt}\mathbb{E}[\|Z_t\|^{2}]\leq &
\mathbb{E}[\|Z_0\|^{2}]- \left( 2(\kappa+C_G-L_{\tilde{G}}  )-K\right)\int_0^t e^{Ks} \mathbb{E}[\|Z_s\|^{2}]ds\\
&+\int_0^t  e^{Ks} \mathbb{E}\left[\|\sigma(s,\tilde X_s)-\eta(s,Y_s)\|_F^2\right]ds\notag\\
\leq & \mathbb{E}[\|Z_0\|^{2}]- \left( 2(\kappa+C_G-L_{\tilde{G}} - (1+\epsilon)L_{\sigma}^2/2 )-K\right) \int_0^te^{Ks} \mathbb{E}[\|Z_s\|^2]ds\notag\\
&+(1+1/\epsilon) \int_0^te^{Ks} \mathbb{E}\left[\|\sigma(s,Y_s)-\eta(s,Y_s)\|_F^2\right]ds\notag 
\end{align}
for all $\epsilon>0$.  With $\epsilon,\kappa$ chosen according to Assumption \ref{assump:main_assumptions_1}\,.\,\ref{assump:epsilon_lambda} and letting $K=K_{\kappa,\epsilon}$   we have
\begin{align}\label{eq:E_Z_2_bound}
\mathbb{E}[\|Z_t\|^{2}]\leq&\mathbb{E}[\|Z_0\|^{2}]e^{-K_{\kappa,\epsilon}t}
+(1+1/\epsilon)\mathbb{E}\left[\int_0^te^{-K_{\kappa,\epsilon}(t-s)}\|\sigma(s,Y_s)-\eta(s,Y_s)\|_F^2ds\right]\,.
\end{align}
  In particular, we have the following bound in terms of the uniform norm of $\sigma-\eta$:
\begin{align}\label{eq:E_Z_2_bound_uniform}
\mathbb{E}[\|Z_t\|^{2}]\leq&\mathbb{E}[\|Z_0\|^{2}]e^{-K_{\kappa,\epsilon}t}+\frac{(1+1/\epsilon)\|\sigma-\eta\|_{F,\infty}^2}{K_{\kappa,\epsilon}}(1-e^{-K_{\kappa,\epsilon}t})\,.
\end{align}
\begin{remark}\label{remark:alternative_E_Z_2_Bounds}
One can alternatively utilize \eqref{eq:E_Z_2_bound} in conjunction with   bounds on $Y_t$ (derived via similar techniques to the above)   to obtain bounds that leverage $(t,x)$-dependent bounds on $\sigma-\eta$, though we will not pursue that direction here.   Similar remarks apply to the $\sigma$ and $\eta$ dependent terms elsewhere in this work.
\end{remark}

While the $n=1$ case \eqref{eq:E_Z_2_bound_uniform}   explicitly enters into our sensitivity bounds,   to justify the use of Girsanov's theorem in Section \ref{sec:Girsanov_KL} via a Novikov-type condition we will also require  sufficient control on   $\mathbb{E}[\|Z_t\|^{2n}]$   for all $n\in\mathbb{Z}^+$.   To that end,  we take the expected value of  \eqref{eq:Z_2n_Ito}  and use the martingale property to compute
\begin{align} 
&e^{Kt}\mathbb{E}[\|Z_t\|^{2n}]\leq\mathbb{E}[\|Z_0\|^{2n}]- \left( 2n(\kappa+C_G-L_{\tilde{G}}  )-K\right)\int_0^t e^{Ks} \mathbb{E}[\|Z_s\|^{2n}]ds\\
&+\int_0^t  2n(n-1)e^{Ks}\mathbb{E}[\|Z_s\|^{2(n-2)} \|(\sigma(s,\tilde X_s)-\eta(s,Y_s))^T Z_s\|^2]ds\notag\\
&+\int_0^t ne^{Ks}\mathbb{E}[\|Z_s\|^{2(n-1)}\|\sigma(s,\tilde X_s)-\eta(s,Y_s)\|_F^2]ds\notag\\
\leq &\mathbb{E}[\|Z_0\|^{2n}]-(2n(\kappa+C_G -L_{\tilde{G}}) -K)\int_0^t e^{Ks} \mathbb{E}[\|Z_s\|^{2n}]ds\notag\\
&+n^2(2\sup_{s,x,y}\|\sigma(s,x)-\eta(s,y)\|_2^2 +\sup_{s,x,y}\|\sigma(s,x)-\eta(s,y)\|_F^2)\!\int_0^t e^{Ks}\mathbb{E}[\|Z_s\|^{2(n-1)}]ds\,.\notag
\end{align}
Note that here    a much rougher bound on the diffusion terms will suffice. Letting $K= 2n(\kappa+C_G-L_{\tilde{G}})$ we obtain  $f_n(t)\leq b_n+a_n\int_0^t  e^{\tau s}f_{n-1}(s)ds$ 
for all $n\in\mathbb{Z}^+$, $t\geq0$, where
\begin{align}
&f_n(t)\coloneqq e^{n\tau t}\mathbb{E}[\|Z_t\|^{2n}]\,,\,\,\,\,\tau\coloneqq 2(\kappa+C_G-L_{\tilde{G}})\,,\,\,\,\,b_n\coloneqq\mathbb{E}[\|Z_0\|^{2n}]\,,\\
&a_n\coloneqq n^2(2\sup_{s,x,y}\|\sigma(s,x)-\eta(s,y)\|_2^2 +\sup_{s,x,y}\|\sigma(s,x)-\eta(s,y)\|_F^2)\coloneqq n^2 D_{\sigma,\eta}\,.
\end{align}
 Functions that satisfy such a recursive sequence of inequalities can be bounded  as follows.
\begin{lemma}\label{lemma:f_n_bound}
    Suppose we have $a_n,b_n\in[0,\infty)$, $T,\tau\in(0,\infty)$, and measurable $f_n:[0,T]\to[0,\infty)$, $n\in\mathbb{Z}^+$ that satisfy
    \begin{align}
        f_n(t)\leq   b_n+a_n\int_0^t e^{\tau s}f_{n-1}(s)ds
    \end{align}
    for all $n\in\mathbb{Z}^+$, $t\in[0,T]$, where $f_0\coloneqq 1$.  Then for all $t\in[0,T]$, and defining $b_0\coloneqq 1$, we have
\begin{align}\label{eq:f_n_bound}
    f_n(t)\leq \sum_{j=0}^n \frac{\prod_{k=j+1}^n a_k}{(n-j)! \tau^{n-j}} b_j  e^{(n-j)\tau t} \,.
\end{align}
\end{lemma}
\begin{proof}  
It is straightforward to check that equality holds when $n=0$.  Supposing the inequality holds for $n$ we have
\begin{align}
f_{n+1}(t)
\leq&  b_{n+1}+a_{n+1}\int_0^t e^{\tau s}\left(\sum_{j=0}^n \frac{\prod_{k=j+1}^n a_k}{(n-j)! \tau^{n-j}} b_j  e^{(n-j)\tau s}\right)ds\notag\\
=&  b_{n+1}+\sum_{j=0}^n \frac{\prod_{k=j+1}^{n+1} a_k}{(n-j)! \tau^{n-j}} b_j  \int_0^t e^{(n+1-j)\tau s}ds\notag\\
\leq &  b_{n+1}+\sum_{j=0}^n \frac{\prod_{k=j+1}^{n+1} a_k}{(n+1-j)! \tau^{n+1-j}} b_j  e^{(n+1-j)\tau t} 
=\sum_{j=0}^{n+1} \frac{\prod_{k=j+1}^{n+1} a_k}{(n+1-j)! \tau^{n+1-j}} b_j  e^{(n+1-j)\tau t}\notag\,.
\end{align}
Therefore the claimed inequality holds for $n+1$ and so the claim holds for all $n$ by induction.
\end{proof}
Applying Lemma \ref{lemma:f_n_bound} together with Assumption \ref{assump:main_assumptions_1}\,.\,\ref{assump:ICS}   yields 
\begin{align} 
   e^{n\tau t}\mathbb{E}[\|Z_t\|^{2n}]\leq& \sum_{j=0}^n \frac{\prod_{k=j+1}^n a_k}{(n-j)! \tau^{n-j}} b_j  e^{(n-j)\tau t} \\
\leq &    \alpha n!\sum_{j=0}^n \frac{n!}{j!(n-j)!}  \beta^j (e^{\tau t} D_{\sigma,\eta}/\tau)^{n-j}
=  \alpha n!(\beta+e^{\tau t} D_{\sigma,\eta}/\tau)^n\,,\notag
\end{align}
where we used Assumption \ref{assump:main_assumptions_1}\,.\,\ref{assump:ICS}. Therefore we can conclude
\begin{align}\label{eq:E_Z_2n_final_bound}
\mathbb{E}[\|Z_t\|^{2n}]\leq& \alpha n!\left(\beta e^{-2(\kappa+C_G-L_{\tilde{G}})t}+ \frac{D_{\sigma,\eta}}{2(\kappa+C_G-L_{\tilde{G}})} \right)^n\,\,\,\, \text{ for all $n\in\mathbb{Z}^+$, $t\geq 0$.}
\end{align}

In addition to  bounds on  $Z_t$, we will   require the following bounds on the moments of $\tilde X_t$. 
\begin{align}\label{eq:E_X_2_bound}
    &\mathbb{E}[\|\tilde X_t\|^{2}] \leq \mathbb{E}[\| X_0\|^{2}] e^{-\tilde{K}_{\kappa,\delta}t}\\
& +\left(2A_G+ \|\sigma\|_{F,\infty}^2    +\frac{\kappa}{\delta } \left(\mathbb{E}[\|Z_0\|^{2}] +\frac{(1+1/\epsilon)\|\sigma-\eta\|_{F,\infty}^2}{K_{\kappa,\epsilon}} \right) \right)\frac{1-e^{-\tilde{K}_{\kappa,\delta}t}}{\tilde{K}_{\kappa,\delta}}\notag
\end{align}
and,  for all $t>0$, $n\in\mathbb{Z}^+$,
\begin{align}\label{eq:f_n_bound_X_tilde_final}
\mathbb{E}[\|\tilde X_t\|^{2n}]
\leq&     n!  \left(\tilde{\alpha} + \alpha /2  \right) \left(\frac{D_{\sigma,G}  }{ B_G}+\max\left\{\tilde{\beta} ,  (2\kappa/B_G)^{2 }  \left(\beta  + \frac{D_{\sigma,\eta}}{2(\kappa+C_G-L_{\tilde{G}})} \right)\right\}\right)^n\notag\,.
\end{align}
The proofs are extremely similar to the above calculations; see Appendix \ref{app:tilde_X_2n_bound}   for details.

\subsubsection{Observables Under   $\tilde{X}_t$  Compared to $Y_t$}
We are now prepared  to show that observables of   $\tilde{X}_t$ and $Y_t$ are close to one another in several cases.  First we consider  time-averaged and discounted observables, defined in terms of a function $h(t,x)$ satisfying the following.
\begin{assumption}\label{assump:h_assump1} 
Let   $h:[0,\infty)\times\mathbb{R}^d\to\mathbb{R}$ be  continuous and suppose there exists  $L_h,M_h\geq 0$ such that
\begin{align}
    |h(t,x)-h(t,y)|\leq L_h(1+M_h(\|x\|+\|y\|))\|x-y\| \,\,\,\,\,\,\text{for all $x,y\in\mathbb{R}^d$, $t\in[0,\infty)$.}
\end{align}
\end{assumption}
In the time-averaged case, the above computations imply the following.
 \begin{lemma}\label{lemma:stage1_result_time_avg}
Under Assumptions \ref{assump:main_assumptions_1} and \ref{assump:h_assump1}, for all $T>0$ we have
\begin{align}
&\left|\mathbb{E}\left[\frac{1}{T}\int_0^T h(t,\tilde X_t)dt\right]-\mathbb{E}\left[\frac{1}{T}\int_0^T h(t,Y_t)dt\right]\right|  
\leq\frac{ L_h(1+1/\epsilon)^{1/2}}{K_{\kappa,\epsilon}^{1/2}}(  1+M_h D_{1,T}) \|\sigma-\eta\|_{F,\infty}\\
& +\frac{2L_h(1-e^{-K_{\kappa,\epsilon}T/2})}{K_{\kappa,\epsilon}} (1+M_hD_{2,T}) \frac{\mathbb{E}[\|X_0-Y_0\|^{2}]^{1/2}}{T} \notag\,,\\
&D_{1,T}\coloneqq \frac{ (1+1/\epsilon)^{1/2}\|\sigma-\eta\|_{F,\infty}}{K_{\kappa,\epsilon}^{1/2}} +  \frac{4(1-e^{-\tilde{K}_{\kappa,\delta}T/2})}{\tilde{K}_{\kappa,\delta}T} \mathbb{E}[\| X_0\|^{2}]^{1/2}\label{eq:D1_def}\\
&+ \frac{2}{\tilde{K}_{\kappa,\delta}^{1/2}}\left(2A_G+ \|\sigma\|_{F,\infty}^2    +\frac{\kappa}{\delta } \left(\mathbb{E}[\|X_0-Y_0\|^{2}] +\frac{(1+1/\epsilon)\|\sigma-\eta\|_{F,\infty}^2}{K_{\kappa,\epsilon}} \right) \right)^{1/2}\,,\notag\\
&D_{2,T}\coloneqq   \frac{2}{ \tilde{K}_{\kappa,\delta}^{1/2}}\left(2A_G+ \|\sigma\|_{F,\infty}^2    +\frac{\kappa}{\delta } \left(\mathbb{E}[\|X_0-Y_0\|^{2}] +\frac{(1+1/\epsilon)\|\sigma-\eta\|_{F,\infty}^2}{K_{\kappa,\epsilon}} \right) \right)^{1/2} \label{eq:D2_def}\\
&+  \frac{  (1-e^{-K_{\kappa,\epsilon}T})}{2(1-e^{-K_{\kappa,\epsilon}T/2})}\mathbb{E}[\|Z_0\|^{2}]^{1/2}+ \frac{2\left(1-e^{-(\tilde{K}_{\kappa,\delta}+K_{\kappa,\epsilon})T/2}\right)}{(1+\tilde{K}_{\kappa,\delta}/K_{\kappa,\epsilon}) (1-e^{-K_{\kappa,\epsilon}T/2})}\mathbb{E}[\| X_0\|^{2}]^{1/2}  \,.\notag
\end{align}
\end{lemma}
\begin{remark}
In particular, if  initial conditions  for $Y_t$ and ${X}_t$ are  the same then Lemma \ref{lemma:stage1_result_time_avg} implies an $O(\|\sigma-\eta\|_{F,\infty})$ bound, uniform in $T>0$.  If $X_0-Y_0\neq 0$ then its contribution to the bound is $O(1/T)$ as $T\to\infty$. 
\end{remark}
\begin{proof}
Using the Cauchy-Schwarz inequality we can compute
\begin{align}
&\left|\mathbb{E}\left[\frac{1}{T}\int_0^T h(t,\tilde X_t)dt\right]-\mathbb{E}\left[\frac{1}{T}\int_0^T h(t,Y_t)dt\right]\right|\\
\leq& \frac{L_h}{T}\int_0^T \mathbb{E}\left[  (1+M_h(\|\tilde X_t\|+\|Y_t\|))\|Z_t\|\right] dt\notag\\
 \leq& \frac{L_h}{T}\int_0^T \mathbb{E}[ \|Z_t\|^2]^{1/2}dt+ \frac{ 2L_hM_h}{T}\int_0^T \mathbb{E}[ \|\tilde X_t\|^2]^{1/2}\mathbb{E}[\|Z_t\|^2]^{1/2}dt+ \frac{L_hM_h }{T}\int_0^T\mathbb{E}[ \|Z_t\|^2]   dt\notag\,.
\end{align}
The result then follows from applying the bounds \eqref{eq:E_Z_2_bound_uniform}    and  \eqref{eq:E_X_2_bound}    along with sub-additivity of the square root.
\end{proof}

Next we consider the case of discounted observables; for simplicity, and because we view it to be the primary case of interest in such applications,  we assume the initial conditions are equal.  To avoid making additional integrability assumptions (regarding the limit as $T\to\infty$), we phrase the result as follows.
\begin{lemma}\label{lemma:stage1_result_discount}
In addition to Assumptions \ref{assump:main_assumptions_1} and \ref{assump:h_assump1}, assume that $\rho(dt)$ is a finite positive measure on $[0,\infty)$ and $X_0=Y_0$.  Then, letting $\|\rho\|_{\mathrm{TV}}$ denote the total variation of $\rho$, we have
\begin{align}
&\sup_{T\geq 0}\left|\mathbb{E}\left[\int_0^T h(t,\tilde X_t)\rho(dt)\right]-\mathbb{E}\left[ \int_0^T h(t,Y_t)\rho(dt) \right]\right| \\
\leq&\frac{L_h(1+1/\epsilon)^{1/2}\|\sigma-\eta\|_{F,\infty}}{K_{\kappa,\epsilon}^{1/2}}
\left(\|\rho\|_{\mathrm{TV}}+M_h\tilde{D}\right)\notag\,,\\
&\tilde{D}\coloneqq  2\mathbb{E}[\| X_0\|^{2}]^{1/2}\int_0^\infty e^{-\tilde{K}_{\kappa,\delta}t/2}\rho(dt)+\frac{(1+1/\epsilon)^{1/2}\|\rho\|_{\mathrm{TV}}\|\sigma-\eta\|_{F,\infty}}{K_{\kappa,\epsilon}^{1/2}}\label{eq:tilde_D_def} \\
&+\frac{2}{\tilde{K}_{\kappa,\delta}^{1/2}}\left(2A_G+ \|\sigma\|_{F,\infty}^2    +    \frac{\kappa}{\delta}\frac{(1+1/\epsilon)\|\sigma-\eta\|_{F,\infty}^2}{K_{\kappa,\epsilon}}  \right)^{1/2}\|\rho\|_{\mathrm{TV}}  \,.\notag
\end{align}
\end{lemma}
\begin{proof}
The proof is nearly identical to that of Lemma \ref{lemma:stage1_result_time_avg}; some simplification occurs due to the assumption that $X_0=Y_0$.
\end{proof}

Finally, we consider the observables that are related to the linear parabolic PDE  \eqref{eq:Feynman_Kac_pde_intro}. 
   \begin{assumption}\label{assump:PDE}
Let  $h:[0,\infty)\times \mathbb{R}^{d}\to\mathbb{R}$ be continuous and $L_h$-Lipschitz in $x$ for all $t$   and let $f:\mathbb{R}^d\to\mathbb{R}^d$ be $L_f$-Lipschitz.
\end{assumption}

\begin{lemma}\label{lemma:Feynman-Kac_X_tilde_Y_bound}
Under Assumptions \ref{assump:main_assumptions_1} and \ref{assump:PDE}, let $\{(\tilde{X}^{t,x}_s,Y^{t,x}_s)\}_{s\geq t}$ be the solution to  \eqref{eq:tilde_X}\,-\,\eqref{eq:Y} started at position $(x,x)\in\mathbb{R}^{d+d}$ at time $t\geq 0$. Given $T>0$, define  
\begin{align}
\tilde{u}_T(t,x)\coloneqq&\mathbb{E}\left[f(\tilde{X}^{t,x}_T)  +\int_t^T h(s,\tilde{X}_s^{t,x})   ds\right]\,,\,\,\,\,\,\,
v_T(t,x)\coloneqq\mathbb{E}\left[f(Y^{t,x}_T) +\int_t^T h(s,Y_s^{t,x}) ds\right]\,.\notag
\end{align}
For all $t\in[0,T]$, $x\in\mathbb{R}^d$ we have
\begin{align}
|\tilde{u}_T(t,x)-v_T(t,x)|
\leq&(  (T-t) L_h + L_f)\frac{(1+1/\epsilon)^{1/2}\|\sigma-\eta\|_{F,\infty}}{K_{\kappa,\epsilon}^{1/2}}\,.
\end{align}
\end{lemma}
\begin{remark}
One could also consider locally Lipschitz functions, as in Assumption \ref{assump:h_assump1}, but for simplicity of the presentation we only consider the Lipschitz case.
\end{remark}
\begin{proof}
Noting that the SDEs for $\tilde{X}_s^{t,x}$ and $Y_s^{t,x}$ also satisfy Assumption \ref{assump:main_assumptions_1} (simply with a change in the initial time, which does not impact the value of any of the constants defined therein), we can apply    Lemma \ref{lemma:stage1_result_time_avg} with $X_0=x=Y_0$ and $M_h=0$ to get
\begin{align}
\left|\mathbb{E}\left[\int_t^T h(s,\tilde X^{t,x}_s)ds\right]-\mathbb{E}\left[\int_t^T h(t,Y^{t,x}_s)ds\right]\right| 
\leq&\frac{  (T-t) L_h(1+1/\epsilon)^{1/2}}{K_{\kappa,\epsilon}^{1/2}}  \|\sigma-\eta\|_{F,\infty}  
\end{align}
for all $0\leq t\leq T$.  Using the Lipschitz property of $f$ along with  \eqref{eq:E_Z_2_bound_uniform} we can compute
\begin{align}
\left|\mathbb{E}[f(\tilde{X}^{t,x}_T)  ]- \mathbb{E}\left[f(Y^{t,x}_T)  \right]\right|\leq&L_f\frac{(1+1/\epsilon)^{1/2}\|\sigma-\eta\|_{F,\infty}}{K_{\kappa,\epsilon}^{1/2}}\,.\notag
\end{align}
 Combining these we obtain the claimed bound.
\end{proof}

\subsection{Information-Theoretic UQ Bound Comparing $\tilde{X}_t$ and $X_t$}
We now proceed to show that observables of  $\tilde{X}_t$ and $X_t$ are close to one another by combining   information-theoretic UQ bounds   with functional inequalities, an approach originally developed in \cite{Birrell2020}. This approach requires that the two SDEs being compared  share the same diffusion term.  Hence  the first stage of our method provides the key to transforming the  problem into one that is within reach of such  techniques.  In this section we   generalize the method from  \cite{Birrell2020} to cover time-dependent observables by using the bounds on the Feynman-Kac semigroup with time-dependent potential  derived in  \cite{birrell2026quasistatic}.

Start by  applying the   information-theoretic UQ bound given in Section 2.2 of \cite{dupuis2016path}    to an observable (measurable function)  on   path-space, $g:C([0,T],\mathbb{R}^d)\to\mathbb{R}$, and the distributions of  $\{\tilde{X}_t\}_{t\in[0,T]}$ and $\{X_t\}_{t\in[0,T]}$; we denote  the corresponding distributions by $P_{\tilde{X}|_{[0,T]}}$ and $P_{{X}|_{[0,T]}}$  respectively and expectations with respect to these will be denoted by $E_{\tilde{X}|_{[0,T]}}$ and  $E_{{X}|_{[0,T]}}$.  With this notation, the UQ bound is given by the following.
\begin{lemma}\label{lemma:UQ_bound_general}
Under Assumption \ref{assump:main_assumptions_1},  for all $g\in L^1(P_{{X}|_{[0,T]}})$ we have
\begin{align}\label{eq:UQ_bound_general}
&\pm\left(E_{{\tilde{X}|_{[0,T]}}}[g]-E_{{{X}|_{[0,T]}}}[g]\right)\\
\leq &\inf_{c>0}\left\{\frac{1}{c}\log E_{{{X}|_{[0,T]}}}\left[\exp(\pm c(g-E_{{{X}|_{[0,T]}}}[g]))\right]+\frac{1}{c}\mathrm{KL}(P_{\tilde{X}|_{[0,T]}}\|P_{{X}|_{[0,T]}})\right\}\,,\notag
\end{align}
where the bound holds regardless of the definition of $E_{{\tilde{X}|_{[0,T]}}}[g]$ in the $\infty-\infty$ case. 
\end{lemma}
In \eqref{eq:UQ_bound_general},  $\mathrm{KL}$ denotes the KL-divergence, i.e., relative entropy, defined by $\mathrm{KL}(Q\|P)\coloneqq E_{Q}[\log(dQ/dP)]$ if $Q\ll P$ and $\mathrm{KL}(Q\|P)\coloneqq\infty$ otherwise.

The objective of the minimization on the right-hand sides of  \eqref{eq:UQ_bound_general}   consists of two terms:
\begin{enumerate}[(a)]
\item The first involves the centered MGF of the observable, $g$.  For this term we will employ standard tail-bounds, i.e., sub-Gaussian and sub-exponential. The case of time-averaged observables requires additional care to ensure the bounds behave well as $T\to\infty$. We address this using $\log$-Sobolev inequalities, as discussed in Section \ref{sec:functional_ineq}.  This technique  also leads to sharper sensitivity bounds for parabolic PDEs.
\item The second term involves the relative entropy between the distributions on path-space of the solutions to the SDEs; in particular, it does not depend on the choice of observable.  We discuss  this term further in Section \ref{sec:Girsanov_KL}.
\end{enumerate}

\subsubsection{MGF and UQ Bound for Time-Averaged Observables via $\log$-Sobolev Inequalities}\label{sec:functional_ineq}
Here we derive a bound on the centered MGF of time-averaged observables  which is well-behaved in the limit as $T\to\infty$.  We will work under the following additional   assumptions.
\begin{assumption}\label{assump:main_assumptions_2}
Assume the following.
\begin{enumerate}
\item    \label{assump:inv_dist}   Suppose the SDE \eqref{eq:X} for $X_t$ is time-homogeneous and has an invariant distribution $\mu_{X}^*$ and the initial condition, $X_0$,  is distributed as $\mu_{X}^*$.
\item \label{assump:log_Sob} Suppose the generator of the SDE  \eqref{eq:X}, given by
\begin{align}\label{eq:gen_def}
A_X[g](x)\coloneqq\frac{1}{2}\sum_{i,j,k}\sigma^i_k(x)\sigma^j_k(x)\partial_{i}\partial_j g(x)+\sum_iF^i(x)\partial_ig(x)\,,
\end{align}
satisfies a $\log$-Sobolev inequality  with respect to $\mu_{X}^*$ with parameter $C_*\in(0,\infty)$, i.e.,
\begin{align}\label{eq:log_Sob_def}
\int g^2\log(g^2)d\mu_{X}^*\leq -C_* \int  A_X[g]g d\mu_{X}^*\,\,\text{ for all }g\in C^2_{pb}(\mathbb{R}^d) \text{ with }\|g\|_{L^2(\mu_{X}^*)}=1\,,
\end{align}
where we   let $C^2_{pb}(\mathbb{R}^d)$ denote the set of $C^2$ real-valued functions on $\mathbb{R}^d$ whose zeroth, first, and second derivatives are polynomially bounded. 
\end{enumerate}
\end{assumption}
\begin{remark}
A number of prior works have used functional inequalities to derive concentration inequalities for stochastic systems  \cite{WU2000435, cattiaux_guillin_2008, doi:10.1137/S0040585X97986667,Birrell2020,birrell2026quasistatic}; our usage   of functional inequalities bears some resemblance to these works, though our end goal differs. We utilize $\log$-Sobolev inequalities  \cite{Gross1975,rothaus1978lower,Bakry1984,bakry1997sobolev,carlen2004logarithmic,bakry2013analysis} due to their ability to handle unbounded potentials.   Note that some authors define $\log$-Sobolev inequalities in terms of  the reciprocal of our parameter $C_*$ and/or differing by factor of $2$. Here we follow the definition from, e.g., \cite{WU2000435}.   
\end{remark}
\begin{remark}
The classical result from \cite{Bakry1984} produces  a $\log$-Sobolev inequality with explicit constant $C_*$ under the following conditions: If $F(x)=-\nabla V(x)$   and the noise is additive then a lower bound on the Hessian, $D_x^2V(x)\geq 2C_*^{-1}I$ for some $C_*>0$, implies a $\log$-Sobolev inequality  with constant $C_*$.  A $\log$-Sobolev inequality   can also be obtained when a bounded perturbation is added  to such a potential; see, e.g.,  Proposition 5.1.6 in \cite{bakry2013analysis}.
 \end{remark}

To bound the MGF term in \eqref{eq:UQ_bound_general} for the observable $g(\gamma)=\int_0^T h(t,\gamma_t)dt$ we utilize the approach developed in \cite{birrell2026quasistatic} for bounding the Feynman-Kac semigroup with a time-dependent potential. Specifically,     we require a variant of Theorem 1 from \cite{birrell2026quasistatic}. The setting under consideration here allows for a streamlined proof; for completeness, we provide the details in Appendix \ref{app:Feynman_Kac_bound}.   The following two lemmas further specialize this result to two important cases; sub-Gaussian and sub-exponential observables, though other forms of tail behavior     could  similarly be analyzed.

 First we consider sub-Gaussian  observables;   given $\sigma_{\mathrm{SG}}\in(0,\infty)$, recall  that a random variable $g$ is $\sigma_{\mathrm{SG}}$-sub-Gaussian with respect to a distribution $Q$   if
\begin{align}\label{eq:sub_Gauss_def}
E_Q\left[ \exp\left(\lambda (g-E_Q[g])\right)\right]\leq \exp\left(\lambda^2\sigma_{\mathrm{SG}}^2/2\right) \, \text{ for all }\,\lambda\in\mathbb{R}\,.
\end{align}
  In particular, this case applies to bounded $h$;  see, e.g., Chapter 2 in \cite{vershynin2018high} for background.
\begin{lemma}\label{lemma:MGF_bound_time_avg_subGauss}
 Under Assumptions \ref{assump:main_assumptions_1}, \ref{assump:h_assump1} and \ref{assump:main_assumptions_2},  suppose also that. for all $t\in[0,T]$, $h_t$ is $\sigma_{h_t}$-sub-Gaussian with respect to $\mu_{X}^*$ for some $\sigma_{h_t}\in(0,\infty)$ and $\sigma_{h_t}\in L^2([0,T],dt)$.  Then
\begin{align}
&\log \mathbb{E}\left[\exp\left(\pm c\int_0^T ( h(t,X_t)-E_{\mu_X^*}[h_t])dt \right)\right]\leq \frac{C_* c^2}{2 }\int_0^T \sigma_{h_t}^2 dt \,\,\,\,\,\,\text{ for all }c>0\,.
\end{align}

\end{lemma}
\begin{proof}
By the definition of a sub-Gaussian random variable \eqref{eq:sub_Gauss_def}, we have 
\begin{align}
\int \exp\left(\pm c C_* \hat{h}_t \right) d\mu_{X}^*\leq e^{ c^2  C_*^2 \sigma_{h_t}^2/2}
\end{align}
  for all $c>0$, where $\hat{h}(t,x)\coloneqq h(t,x)-E_{\mu_X^*}[h_t]$. Applying  Theorem 1 from \cite{birrell2026quasistatic} (see also Theorem \ref{thm:MGF_bound} in Appendix \ref{app:Feynman_Kac_bound}) to $\pm c\hat{h}(t,x)$, noting that if $h$ satisfies  Assumption \ref{assump:h_assump1} then so does $\pm c\hat{h}(t,x)$,  and for all $c>0$ we can   compute
\begin{align}
\log \mathbb{E}\left[\exp\left(\pm c\int_0^T \hat{h}(t,X_t)dt \right)\right]\leq& \int_0^TC_*^{-1}\log\left( \int \exp\left(\pm c C_*(  {h}_t- E_{\mu_{X}^*}[ h_t]) \right) d\mu_{X}^*\right)dt\\
\leq& \frac{C_* c^2}{2 }\int_0^T \sigma_{h_t}^2 dt\notag\,.
\end{align}
\end{proof}

Next we consider sub-exponential observables. More  specifically, we consider those satisfying a  Bernstein-type MGF bound. Given $\sigma_B,b\in(0,\infty)$, we say that a random variable $g$ satisfies a $(\sigma_B,b)$-Bernstein MGF bound with respect to a distribution $Q$   if
\begin{align}\label{eq:Bernstein_def}
E_Q\left[ \exp\left(\lambda (g-E_Q[g])\right)\right]\leq \exp\left(\frac{\lambda^2\sigma_B^2}{2(1-b|\lambda|)}\right) \, \text{ for all }\,|\lambda|<1/b\,.
\end{align}
Such bounds  can be proven via the  Bernstein condition;  see Proposition 2.10 in \cite{wainwright2019high}.
\begin{lemma}\label{lemma:MGF_bound_time_avg_subexp}
 Under Assumptions \ref{assump:main_assumptions_1}, \ref{assump:h_assump1} and \ref{assump:main_assumptions_2},  suppose also that for all $t\in[0,T]$, $h_t$ satisfies a $(\sigma_{h_t},b)$-Bernstein MGF bound with respect to $\mu_{X}^*$  with $\sigma_{h_t}\in L^2([0,T],dt)$.  Then
\begin{align} 
&\log \mathbb{E}\left[\exp\left(\pm c \int_0^{T} ( h(t,X_t)-E_{\mu_X^*}[h_t]) dt\right)\right]\leq    \frac{ c^2 C_*}{2(1-c C_*b)} \int_0^T \! \sigma_{h_t}^2dt\,\,\text{ for all }c\in(0,1/(C_*b)).
\end{align}

\end{lemma}
\begin{proof}
Define   $\hat{h}(t,x)\coloneqq h(t,x)-E_{\mu_X^*}[h_t]$. Combining the definition \eqref{eq:Bernstein_def} with Theorem 1 from \cite{birrell2026quasistatic},  for all $T>0$ and all $c\in(0,1/(C_*b))$ we have 
\begin{align}
\log \mathbb{E}\left[\exp\left(\pm c \int_0^{T} \hat{h}(t,X_t) dt\right)\right]\leq&  C_*^{-1}\int_0^T\log\left( \int \exp\left(\pm c C_* (h_t-E_{\mu_X^*}[h_t]) \right) d\mu_{X}^*\right)dt\\
\leq&  \frac{ c^2 C_*}{2(1-c C_*b)} \int_0^T \!\sigma_{h_t}^2dt\notag\,.
\end{align}
\end{proof}

\subsubsection{Relative Entropy Bound via Girsanov's Theorem}\label{sec:Girsanov_KL}

To bound the relative entropy term in \eqref{eq:UQ_bound_general}, we   utilize Girsanov's theorem together with the moment bounds  derived in   Section \ref{sec:Z_2n_bound}. We will require the following additional assumptions on the diffusion for $X_t$.
\begin{assumption}\label{assump:sigma_invert} 
 Assume $\sigma$ is invertible and $\|\sigma^{-1}\|_{2,\infty}\coloneqq \sup_{t\geq 0,x\in\mathbb{R}^d}\|\sigma^{-1}(t,x)\|_2<\infty$, 
 where  $\|\cdot\|_2$ denotes the $\ell^2$ matrix-norm.
\end{assumption}

\begin{lemma}\label{lemma:KL_bound}
Under Assumptions \ref{assump:main_assumptions_1} and \ref{assump:sigma_invert},  for all $T>0$ we have the following bound on the relative entropy between the distributions of $\tilde{X}_t$ and $X_t$ on path-space:
\begin{align}
\mathrm{KL}\left(P_{\tilde{X}|_{[0,T]}}\|P_{{X}|_{[0,T]}}\right)\leq& \frac{1}{2}\int_0^T\mathbb{E}\left[ \left\|\sigma^{-1}(t,\tilde{X}_t)\left(F(t,\tilde{X}_t)-G(t,\tilde{X}_t)+\kappa(\tilde{X}_t-Y_t)\right)\right\|^2\right]dt\,.
\end{align}
\end{lemma}
\begin{proof}

We will apply Girsanov's theorem to the system of SDEs \eqref{eq:tilde_X}\,-\,\eqref{eq:Y} with the change of drift induced by the martingale
\begin{align}\label{eq:N_t_def}
&N_t\coloneqq\exp\left(\int_0^t b(s,\tilde{X}_s,Y_s) \cdot dW_s-\frac{1}{2}\int_0^t \|b(s,\tilde{X}_s,Y_s)\|^2ds\right)\,,\\
&b(t,\tilde{x},y)\coloneqq \sigma^{-1}(t,\tilde{x})(F(t,\tilde{x})-G(t,\tilde{x})+\kappa(\tilde{x}-y))\,.\label{eq:b_vf_def}
\end{align}
This will result in a solution to the   SDE system with the same diffusion and modified drift   
\begin{align}
(t,\tilde{x},y)\mapsto& \begin{pmatrix}
G(t,\tilde{x})-\kappa(\tilde{x}-y)+\sigma(t,\tilde{x})b(t,\tilde{x},y)\\
G(t,y)+\eta(t,y)b(t,\tilde{x},y)
\end{pmatrix}\,,
\end{align}
which simplifies to give the drift for the system \eqref{eq:X}\,-\,\eqref{eq:tilde_Y}.  The latter system has the same diffusion as \eqref{eq:tilde_X}\,-\,\eqref{eq:Y}, hence Girsanov's theorem will imply absolute continuity of the path-space distributions of the   solutions to these two systems and will also allow us compute their relative entropy.

To justify the use of Girsanov's theorem we prove that $N_t$ is a martingale using the variant of Novikov's condition given in Corollary 5.14 on page 199 of \cite{karatzas2014brownian}.  For   $\Delta t>0$ define $t_j=j\Delta t$, $j\in\mathbb{Z}^+$, and use the Cauchy-Schwarz inequality to compute 
\begin{align}
&\mathbb{E}\left[\exp\left(\frac{1}{2}\int_{t_{j-1}}^{t_j}\|b(s,\tilde{X}_s,Y_s)  \|^2 ds\right)\right]\\
\leq &\mathbb{E}\left[\exp\left(2\| \sigma^{-1}\|_{2,\infty}^2\int_{t_{j-1}}^{t_j}\|F(s,\tilde{X}_s)-G(s,\tilde{X}_s)\|^2ds\right)\right]^{1/2}\notag\\
&\times\mathbb{E}\left[ \exp\left(2\kappa^2\| \sigma^{-1}\|_{2,\infty}^2\int_{t_{j-1}}^{t_j}\|Z_s\|^2 ds\right)\right]^{1/2}\notag\,.
\end{align}
We need to prove that both of these terms are finite for appropriately chosen $\Delta t$. 

By Assumption \ref{assump:main_assumptions_1}\,.\,\ref{assump:existence_uniqueness}, $F$ and $G$ are linearly bounded in $x$, uniformly in $t$, hence there exists $\tilde{A}$, $\tilde{B}\geq 0$ such that
\begin{align}
&\mathbb{E}\left[\exp\left(2\| \sigma^{-1}\|_{2,\infty}^2\int_{t_{j-1}}^{t_j}\|F(s,\tilde{X}_s)-G(s,\tilde{X}_s)\|^2ds\right)\right]\\
\leq&e^{\tilde{A} \Delta t}\mathbb{E}\left[\exp\left( \tilde{B}\int_{t_{j-1}}^{t_j} \|\tilde{X}_s\|^2ds\right)\right]\leq e^{\tilde{A} \Delta t}\sum_{n=0}^\infty  \frac{\tilde{B}^n \Delta t^{n-1}}{n!} \int_{t_{j-1}}^{t_j}\mathbb{E}\left[  \|\tilde{X}_s\|^{2n}\right]ds\notag\,, 
\end{align}
where we again used the Cauchy-Schwarz inequality to obtain the last line.
Now apply the bound \eqref{eq:f_n_bound_X_tilde_final}, which holds for all $n$  (note that it trivially holds for $n=0$) to compute
\begin{align}
&\mathbb{E}\left[\exp\left(2\| \sigma^{-1}\|_{2,\infty}^2\int_{t_{j-1}}^{t_j}\|F(s,\tilde{X}_s)-G(s,\tilde{X}_s)\|^2ds\right)\right]\\
\leq &e^{\tilde{A} \Delta t} \left(\tilde{\alpha} + \alpha /2  \right) \sum_{n=0}^\infty     \left(\Delta t \tilde{B}\left(\frac{D_{\sigma,G}  }{ B_G}+\max\left\{\tilde{\beta}  ,  (2\kappa/B_G)^{2 }  \left(\beta  + \frac{D_{\sigma,\eta}}{2(\kappa+C_G-L_{\tilde{G}})} \right)\right\}\right)\right)^n<\infty\notag
\end{align}
for all $j$ when  $\Delta t$ is sufficiently small.

Similarly, we can use the bound   \eqref{eq:E_Z_2n_final_bound} (note that it trivially holds for $n=0$) to compute
\begin{align}
&\mathbb{E}\left[ \exp\left(2\kappa^2\| \sigma^{-1}\|_{2,\infty}^2\int_{t_{j-1}}^{t_j}\|Z_s\|^2 ds\right)\right]\\
\leq& \alpha\sum_{n=0}^\infty\left(2\Delta t\kappa^2\| \sigma^{-1}\|_{2,\infty}^2  \left(\beta  + \frac{D_{\sigma,\eta}}{2(\kappa+C_G-L_{\tilde{G}})} \right)\right)^n<\infty\notag
\end{align}
for all $j$ when  $\Delta t$ is sufficiently small.

Hence, for $\Delta t$ sufficiently small we can conclude that
\begin{align}
\mathbb{E}\left[\exp\left(\frac{1}{2}\int_{t_{j-1}}^{t_j}\|b(s,\tilde{X}_s,Y_s)  \|^2 ds\right)\right]<\infty
\end{align}
for all $j$. 
 Finiteness of these expectations along with the fact that $t_j\nearrow\infty$ as $j\to\infty$ implies  \eqref{eq:N_t_def}
is a   martingale (again, see Corollary 5.14 on page 199 of \cite{karatzas2014brownian}). Hence we are justified in using Girsanov's theorem,  e.g., see Section 3.5 in \cite{karatzas2014brownian}, which  implies that $(\tilde{X}_t,Y_t)|_{t\in[0,T]}$ is a weak solution to the system  \eqref{eq:X}-\eqref{eq:tilde_Y}  with respect to the measure $d\widetilde{\mathbb{P}}_T\coloneqq N_T d\mathbb{P}$ (and driven by an appropriate $\widetilde{\mathbb{P}}_T$-Wiener process, $\tilde{W}_t$, in place of $W_t$).  As $(X_t,\tilde{Y}_t)$ is also a solution to this SDE (with respect to  the $\mathbb{P}$-Wiener process $W_t$) and  the initial distributions agree, uniqueness in law  implies equality of the corresponding distributions on path-space up to time $T$, i.e.,
\begin{align}\label{eq:unique_in_dist}
((X,\tilde{Y})|_{[0,T]})_\#\mathbb{P}=((\tilde{X},Y)|_{[0,T]})_\#\widetilde{\mathbb{P}}_T\,,
\end{align}
where  $\phi_\#\nu$ denotes the pushforward of a measure $\nu$ by a measurable map $\phi$.

The distributions $P_{\tilde{X}|_{[0,T]}}$ and $P_{{X}|_{[0,T]}}$ are the first marginals of   $((\tilde{X},Y)|_{[0,T]})_\#\mathbb{P}$ and $((X,\tilde{Y})|_{[0,T]})_\#\mathbb{P}$ respectively. The KL-divergence data processing inequality implies that marginalization can only reduce the relative entropy, hence we can compute 
\begin{align}
\mathrm{KL}\left(P_{\tilde{X}|_{[0,T]}}\|P_{{X}|_{[0,T]}}\right)\leq& \mathrm{KL}\left(\mathbb{P}\| \widetilde{\mathbb{P}}_T\right)=\frac{1}{2}\int_0^T\mathbb{E}\left[ \|b(s,\tilde{X}_s,Y_s)\|^2\right]ds\,,
\end{align}
where we used  \eqref{eq:unique_in_dist}, the definitions of $\widetilde{\mathbb{P}}_T$ and the KL-divergence, along with the fact that $\mathbb{E}[\int_0^t \|b(s,\tilde{X}_s,Y_s)\|^2 ds]<\infty$ for all $t$ and hence $\int_0^t b(s,\tilde{X}_s,Y_s)\cdot  dW_s$ is a $\mathbb{P}$-martingale.  Substituting in the definition of $b$, Eq.~\eqref{eq:b_vf_def}, we arrive at the claimed result.
\end{proof}

\section{Conclusion}
In this work we have used a combination of stochastic calculus estimates, functional inequalities, and information-theoretic UQ bounds to derive novel sensitivity bounds on the solutions to  SDEs, allowing for  perturbation of both the diffusion and the drift.  To the best of the author's knowledge, these results are the first of this type that allow for perturbation of the diffusion and which produce results that are well-behaved at large times; in particular, previous information-theoretic sensitivity bounds    only apply to perturbations to the drift.  Several applications of our new method were studied: 1) sensitivity bounds on time-averaged and exponentially discounted observables,  2) bounds on the $1$-Wasserstein distance between invariant distributions of the baseline and perturbed SDEs, 3) sensitivity bounds on solutions to linear parabolic PDEs.   By examining invariant distributions of  OU processes, we showed that our method produces bounds that scale optimally in the size of the perturbation to the diffusion and to the drift.  Currently, the method is restricted to the case of a uniformly-elliptic baseline SDE; the investigation of  more general cases   (e.g., underdamped Langevin dynamics) is left for future work.

\appendix
\section{Additional Proofs}
Below we provide additional  proof details regarding several results discussed in the main text.

\subsection{Theorem \ref{thm:bounds_time_avg}: Bernstein MGF Bound Case}\label{app:time_avg_sub_exp}
In this section we provide details regarding the proof of \eqref{eq:bounds_time_avg_thm_subexp_general}  from Theorem \ref{thm:bounds_time_avg} under Bernstein MGF bound assumptions.

Starting from the UQ bound \eqref{eq:UQ_bound_X_tilde_X_time_avg}, we  apply the Bernstein MGF bound from Lemma \ref{lemma:MGF_bound_time_avg_subexp} to obtain

\begin{align}
&\left|\mathbb{E}\left[\frac{1}{T}\int_0^T h(t,\tilde{X}_t)dt\right]- \frac{1}{T}\int_0^T E_{\mu_X^*}[h_t] dt \right|\\
\leq& \inf_{c\in(0,1/(C_*b)}\left\{\frac{1}{T} \frac{ c C_*}{2(1-c C_*b)} \int_0^T \sigma_{h_t}^2dt+\frac{1}{cT}\mathrm{KL}\left(P_{\tilde{X}|_{[0,T]}}\|P_{{X}|_{[0,T]}}\right)\right\}\,.\notag
\end{align}
The optimization over $c$ can be computed exactly to yield 
\begin{align}
&\left|\mathbb{E}\left[\frac{1}{T}\int_0^T h(t,\tilde{X}_t)dt\right]- \frac{1}{T}\int_0^T E_{\mu_X^*}[h_t] dt \right|\\
\leq& \left(2     C_* \frac{1}{T}\int_0^T \sigma_{h_t}^2dt \frac{1}{T} \mathrm{KL}\left(P_{\tilde{X}|_{[0,T]}}\|P_{{X}|_{[0,T]}}\right)\right)^{1/2} +\frac{C_*b}{T}\mathrm{KL}\left(P_{\tilde{X}|_{[0,T]}}\|P_{{X}|_{[0,T]}}\right)\,.\notag
\end{align}
Combining this with \eqref{eq:time_avg_proof_X_tilde_Y_step} and \eqref{eq:time_avg_proof_KL_step} completes the proof.

 \subsection{Bounds on $\tilde X_t$}\label{app:tilde_X_2n_bound}   

In addition to the   bounds on  the moments of $Z_t=\tilde X_t-Y_t$ that were derived in Section \ref{sec:Z_2n_bound}, we will require bounds on the moments of $\tilde X_t$; we continue to work under 
 Assumption \ref{assump:main_assumptions_1}. Similarly to the computations in Section \ref{sec:Z_2n_bound}, for $\tilde{K}\in\mathbb{R}$, $n\in\mathbb{Z}^+$ we can use It{\^o}'s formula along with Assumption  \ref{assump:main_assumptions_1}\,.\,\ref{assump:G_confining}   to compute
\begin{align}\label{eq:Ito_tilde_X_app}
    &e^{\tilde{K}t}\|\tilde X_t\|^{2n}\\
=&\|\tilde X_0\|^{2n}+\int_0^t \tilde{K}e^{\tilde{K}s} \|\tilde X_s\|^{2n}ds +2n\int_0^te^{\tilde{K}s}\|\tilde X_s\|^{2(n-1)}\tilde X_s\cdot\sigma(s,\tilde X_s)dW_s\notag\\
&+2n\int_0^te^{\tilde{K}s}\|\tilde X_s\|^{2(n-1)}\tilde X_s\cdot (G(s,\tilde X_s)-\kappa(\tilde X_s-Y_s))ds\notag\\
&+\frac{1}{2}\sum_{i,j,k}\int_0^t e^{\tilde{K}s}\left(2n(n-1)\|\tilde X_s\|^{2(n-2)}2 \tilde X_s^j\tilde X_s^i+2n\|\tilde X_s\|^{2(n-1)}\delta_{ij}\right)   \sigma_k^i(s,\tilde X_s)\sigma_k^j(s,\tilde X_s) ds\notag\\
\leq&\|\tilde X_0\|^{2n}+\int_0^t \tilde{K}e^{\tilde{K}s} \|\tilde X_s\|^{2n}ds+2n\int_0^te^{\tilde{K}s}\|\tilde X_s\|^{2(n-1)}\tilde X_s\cdot\sigma(s,\tilde X_s)dW_s\notag\\
&+2n\int_0^te^{\tilde{K}s}\|\tilde X_s\|^{2(n-1)}( A_G-B_G\|\tilde{X}_s\|^2+\kappa \|\tilde X_s\|\|Z_s\|)ds\notag\\
&+\frac{1}{2}\sum_{i,j,k}\int_0^t e^{\tilde{K}s}\left(2n(n-1)\|\tilde X_s\|^{2(n-2)}2 \tilde X_s^j\tilde X_s^i+2n\|\tilde X_s\|^{2(n-1)}\delta_{ij}\right)   \sigma_k^i(s,\tilde X_s)\sigma_k^j(s,\tilde X_s) ds\notag\,.
\end{align}
Moreover,
\begin{align}
    \mathbb{E}\left[\int_0^t\left(e^{\tilde{K}s}\|\tilde X_s\|^{2(n-1)}\|\sigma^T(s,\tilde X_s)\tilde X_s\|\right)^2ds\right]<\infty\,,
\end{align}
for all $t$, which implies $\int_0^te^{\tilde{K}s}\|\tilde X_s\|^{2(n-1)}\tilde X_s\cdot\sigma(s,\tilde X_s)dW_s$ is a martingale.

Specializing   \eqref{eq:Z_2n_Ito} to $n=1$, taking the expectation of both sides and using the martingale property, we   obtain
\begin{align}
    e^{\tilde{K}t}\mathbb{E}[\|\tilde X_t\|^{2}]
\leq &\mathbb{E}[\|\tilde X_0\|^{2}]+\int_0^t \tilde{K}e^{\tilde{K}s} \mathbb{E}[\|\tilde X_s\|^{2}]ds+\int_0^t e^{\tilde{K}s}\mathbb{E}\left[\|\sigma(s,\tilde X_s)\|_F^2    \right]ds\notag\\
&+2\int_0^te^{\tilde{K}s}(A_G-B_G\mathbb{E}[\|\tilde X_s\|^2])ds+2\kappa\int_0^te^{\tilde{K}s}\mathbb{E}[\|\tilde X_s\|\|Z_s\|]ds \notag\,.
\end{align}
With $\delta$ chosen according to Assumption \ref{assump:main_assumptions_1}\,.\,\ref{assump:delta} and letting $\tilde{K}=\tilde{K}_{\kappa,\delta}$ we can compute
\begin{align}
&    e^{\tilde{K}_{\kappa,\delta}t}\mathbb{E}[\|\tilde X_t\|^{2}]\\
\leq &\mathbb{E}[\|\tilde X_0\|^{2}]+(\tilde{K}_{\kappa,\delta}-2B_G + \kappa\delta)\int_0^t e^{\tilde{K}_{\kappa,\delta}s} \mathbb{E}[\|\tilde X_s\|^{2}]ds\notag\\
&+\int_0^t e^{\tilde{K}_{\kappa,\delta}s}\mathbb{E}\left[\|\sigma(s,\tilde X_s)\|_F^2    \right]ds+2A_G\int_0^te^{\tilde{K}_{\kappa,\delta}s} ds+\kappa\delta^{-1}\int_0^te^{\tilde{K}_{\kappa,\delta}s}\mathbb{E}[\|Z_s\|^2]ds \notag\\
\leq &\mathbb{E}[\|\tilde X_0\|^{2}]+(2A_G+ \|\sigma\|_{F,\infty}^2    )\frac{e^{\tilde{K}_{\kappa,\delta}t}-1}{\tilde{K}_{\kappa,\delta}}+\kappa\delta^{-1}\int_0^te^{\tilde{K}_{\kappa,\delta}s}\mathbb{E}[\|Z_s\|^2]ds \,,\notag
\end{align}
and hence
\begin{align}
    \mathbb{E}[\|\tilde X_t\|^{2}]
 \leq &\mathbb{E}[\|\tilde X_0\|^{2}] e^{-\tilde{K}_{\kappa,\delta}t} +(2A_G+ \|\sigma\|_{F,\infty}^2    )\frac{1-e^{-\tilde{K}_{\kappa,\delta}t}}{\tilde{K}_{\kappa,\delta}}\\
&+\kappa\delta^{-1}e^{-\tilde{K}_{\kappa,\delta}t}\int_0^te^{\tilde{K}_{\kappa,\delta}s}\mathbb{E}[\|Z_s\|^2]ds \,.\notag
\end{align}
Using \eqref{eq:E_Z_2_bound_uniform} and Assumption \ref{assump:main_assumptions_1}\,.\,\ref{assump:ICS} we obtain
\begin{align}\label{eq:E_X_2_bound_app}
    &\mathbb{E}[\|\tilde X_t\|^{2}]\\
 \leq &\mathbb{E}[\| X_0\|^{2}] e^{-\tilde{K}_{\kappa,\delta}t} +\left(2A_G+ \|\sigma\|_{F,\infty}^2    +\frac{\kappa}{\delta } \left(\mathbb{E}[\|Z_0\|^{2}] +\frac{(1+1/\epsilon)\|\sigma-\eta\|_{F,\infty}^2}{K_{\kappa,\epsilon}} \right) \right)\frac{1-e^{-\tilde{K}_{\kappa,\delta}t}}{\tilde{K}_{\kappa,\delta}}\,.\notag
\end{align}

As was the case for  $Z_t$, we will also require control of higher moments in order to justify our use of Girsanov's theorem in Section \ref{sec:Girsanov_KL}. By  taking the expectation of  \eqref{eq:Ito_tilde_X_app} and using the martingale property we can compute
\begin{align} 
    &e^{\tilde{K}t}\mathbb{E}[\|\tilde X_t\|^{2n}]\\
\leq&\mathbb{E}[\|\tilde X_0\|^{2n}]+\int_0^t \tilde{K}e^{\tilde{K}s} \mathbb{E}[\|\tilde X_s\|^{2n}]ds \notag\\
&+2n\int_0^te^{\tilde{K}s}\mathbb{E}[\|\tilde X_s\|^{2(n-1)}( A_G-B_G\|\tilde{X}_s\|^2+\kappa \|\tilde X_s\|\|Z_s\|)]ds\notag\\
&+\frac{1}{2}\sum_{i,j,k}\int_0^t e^{\tilde{K}s}\mathbb{E}[\left(2n(n-1)\|\tilde X_s\|^{2(n-2)}2 \tilde X_s^j\tilde X_s^i+2n\|\tilde X_s\|^{2(n-1)}\delta_{ij}\right)   \sigma_k^i(s,\tilde X_s)\sigma_k^j(s,\tilde X_s)]ds\notag\\
\leq&\mathbb{E}[\|\tilde X_0\|^{2n}]+\int_0^t \tilde{K}e^{ \tilde{K}s} \mathbb{E}[\|\tilde X_s\|^{2n}]ds\notag\\
&+2n\int_0^te^{ \tilde{K}s} ( A_G\mathbb{E}[\|\tilde X_s\|^{2(n-1)}]-B_G\mathbb{E}[\|\tilde X_s\|^{2n}]+\kappa\mathbb{E}[\|\tilde X_s\|^{2n-1}\|Z_s\|])ds\notag\\
&+\left(2n(n-1)\|\sigma\|^2_{2,\infty} +n\|\sigma\|_{F,\infty}^2 \right)\int_0^t e^{ \tilde{K}s} \mathbb{E}\left[\|\tilde X_s\|^{2(n-1)}\right] ds\,.\notag
\end{align}
Using  Young's inequality with $p=2n$, $1/q=1-1/2n=(2n-1)/2n$, for all $\epsilon>0$  we can further compute
\begin{align}
    e^{ \tilde{K}t}\mathbb{E}[\|\tilde X_t\|^{2n}]
\leq&\mathbb{E}[\|\tilde X_0\|^{2n}]+\kappa\epsilon^{-(2n-1)}\int_0^te^{ \tilde{K}s}\mathbb{E}[\|Z_s\|^{2n}]ds\\
&-(2nB_G+\kappa\epsilon-2n\kappa\epsilon- \tilde{K})\int_0^te^{ \tilde{K}s} \mathbb{E}[\|\tilde X_s\|^{2n}]ds\notag\\
&+\left(2n(n-1)\|\sigma\|^2_{2,\infty} +n\|\sigma\|_{F,\infty}^2+2n A_G \right)\int_0^t e^{ \tilde{K}s} \mathbb{E}\left[\|\tilde X_s\|^{2(n-1)}\right] ds\,.\notag
\end{align}
Let $\epsilon=B_G/(2\kappa)$ and $ \tilde{K}=nB_G$.  Fixing $T>0$, for $t\in[0,T]$ the above implies
\begin{align}\label{eq:tilde_X_2n_bound_recursive}
    e^{nB_Gt}\mathbb{E}[\|\tilde X_t\|^{2n}]
\leq&\mathbb{E}[\|\tilde X_0\|^{2n}]+\kappa(2\kappa/B_G)^{2n-1}\int_0^Te^{nB_Gs}\mathbb{E}[\|Z_s\|^{2n}]ds\\
&+n^2D_{\sigma,G}\int_0^t e^{nB_Gs} \mathbb{E}\left[\|\tilde X_s\|^{2(n-1)}\right] ds\,,\notag
\end{align}
where $D_{\sigma,G}\coloneqq  2\|\sigma\|^2_{2,\infty}+\|\sigma\|_{F,\infty}^2+2A_G$.
The bound \eqref{eq:tilde_X_2n_bound_recursive} has the form
\begin{align}
    &f_n(t)\leq b_n+a_n\int_0^t e^{\tau s} f_{n-1}(s) ds\,,\,\,t\in[0,T]\,.
\end{align}
 where   
\begin{align}
f_n(t)\coloneqq &e^{n\tau t}\mathbb{E}[\|\tilde X_t\|^{2n}]\,,\,\,\,   \tau\coloneqq B_G\,,\\
b_n\coloneqq&\mathbb{E}[\|\tilde X_0\|^{2n}]+\kappa(2\kappa/B_G)^{2n-1}\int_0^Te^{nB_Gs}\mathbb{E}[\|Z_s\|^{2n}]ds\,,\\
a_n\coloneqq &n^2D_{\sigma,G}\,.
\end{align}
  Therefore we can  apply Lemma \ref{lemma:f_n_bound}, which yields the bound
\begin{align}\label{eq:f_n_bound_X_tilde}
e^{nB_G t}\mathbb{E}[\|\tilde X_t\|^{2n}]\leq& \sum_{j=0}^n \frac{\prod_{k=j+1}^n a_k}{(n-j)! \tau^{n-j}} b_j  e^{(n-j)\tau t} \\
=& \sum_{j=0}^n \frac{(n!/j!)^2}{(n-j)!} (D_{\sigma,G} e^{ B_G t}/ B_G)^{n-j}b_j  \notag
\end{align}
for all $t\in[0,T]$, where $b_0\coloneqq 1$.  Using Assumption \ref{assump:main_assumptions_1}\,.\,\ref{assump:ICS}  along with \eqref{eq:E_Z_2n_final_bound}, for $n\in\mathbb{Z}^+$  we can compute
\begin{align}
b_n\leq & \tilde{\alpha}\tilde{\beta}^n n!+\kappa \alpha n! (2\kappa/B_G)^{2n-1} \left(\beta  + \frac{D_{\sigma,\eta}}{2(\kappa+C_G-L_{\tilde{G}})} \right)^n\int_0^Te^{nB_Gs}ds \\
\leq &n!\left(\tilde{\alpha}\tilde{\beta}^n + \frac{1}{2}\alpha    \left( (2\kappa/B_G)^{2 }e^{ B_GT} \left(\beta  + \frac{D_{\sigma,\eta}}{2(\kappa+C_G-L_{\tilde{G}})} \right)\right)^n \right) \notag\\
\leq &n!\left(\tilde{\alpha} +  \alpha /2  \right)\max\left\{\tilde{\beta},  (2\kappa/B_G)^{2 }e^{ B_GT} \left(\beta  + \frac{D_{\sigma,\eta}}{2(\kappa+C_G-L_{\tilde{G}})} \right)\right\}^n\notag
\end{align}
 (note that the final line also bounds $b_0$ when $n=0$).  Combining this with \eqref{eq:f_n_bound_X_tilde} we obtain
\begin{align}
&e^{nB_G t}\mathbb{E}[\|\tilde X_t\|^{2n}]\\
\leq& n! \left(\tilde{\alpha} + \alpha /2  \right)   \sum_{j=0}^n \frac{n! }{j!(n-j)!} \left(\frac{D_{\sigma,G} e^{ B_G t}}{ B_G}\right)^{n-j}\notag\\
&\qquad\qquad\qquad\times \max\left\{\tilde{\beta},  (2\kappa/B_G)^{2 }e^{ B_GT} \left(\beta  + \frac{D_{\sigma,\eta}}{2(\kappa+C_G-L_{\tilde{G}})} \right)\right\}^j \notag\\
=&n! \left(\tilde{\alpha} + \alpha /2  \right) \left(\frac{D_{\sigma,G} e^{ B_G t}}{ B_G}+\max\left\{\tilde{\beta},  (2\kappa/B_G)^{2 }e^{ B_GT} \left(\beta  + \frac{D_{\sigma,\eta}}{2(\kappa+C_G-L_{\tilde{G}})} \right)\right\}\right)^n\notag
\end{align}
 for all $t\in[0,T]$. Dividing by $e^{nB_G t}$, letting $t=T$, and recalling that $T>0$ was arbitrary, we find
\begin{align}\label{eq:f_n_bound_X_tilde_final_app}
\mathbb{E}[\|\tilde X_t\|^{2n}]
\leq&  n!  \left(\tilde{\alpha} + \alpha /2  \right) \left(\frac{D_{\sigma,G}  }{ B_G}+\max\left\{\tilde{\beta} e^{-B_G t},  (2\kappa/B_G)^{2 }  \left(\beta  + \frac{D_{\sigma,\eta}}{2(\kappa+C_G-L_{\tilde{G}})} \right)\right\}\right)^n\\
\leq&  n!  \left(\tilde{\alpha} + \alpha /2  \right) \left(\frac{D_{\sigma,G}  }{ B_G}+\max\left\{\tilde{\beta} ,  (2\kappa/B_G)^{2 }  \left(\beta  + \frac{D_{\sigma,\eta}}{2(\kappa+C_G-L_{\tilde{G}})} \right)\right\}\right)^n\notag
\end{align}
 for all $t>0$, $n\in\mathbb{Z}^+$.

\subsection{Bound on the Feynman-Kac Semigroup with a Time-Dependent Potential}\label{app:Feynman_Kac_bound}

In this section we derive a bound on the Feynman-Kac semigroup with a time-dependent potential via a streamlined version of the argument used in \cite{birrell2026quasistatic} to study concentration inequalities for time-inhomogeneous systems; see   Theorem 1 therein.  Specifically, we use the Feynman-Kac formula, i.e., that  (under appropriate assumptions)
\begin{align}
u^f_T(t,x)= \mathbb{E}\left[ f(X^{t,x}_T)\exp\left(\int_t^T h(s, X_s^{t,x})ds\right)\right] 
\end{align}
 solves the PDE
\begin{align} 
&\partial_t u^f_T(t,x)=-A_{X}[u^f_T(t,\cdot)](t,x)-h(t,x)  u^f_T(t,x)\,,\,\,\,\,u^f_T(T,x)=f(x)\,,
\end{align}
 together with a $\log$-Sobolev inequality for the generator, $A_X$. We note that related techniques have also previously been used  to derive concentration inequalities for Markov processes   for   time-independent  $h(x)$; see \cite{WU2000435,doi:10.1137/S0040585X97986667,cattiaux_guillin_2008,birrell2025concentration}.
\begin{theorem}\label{thm:MGF_bound}
 Under Assumptions  \ref{assump:main_assumptions_1}, \ref{assump:h_assump1} and \ref{assump:main_assumptions_2},  for all $T>0$ we have the following bound on the MGF:
\begin{align}\label{eq:MGF_bound}
&\mathbb{E}\left[\exp\left(\int_0^{T} h(t,X_t) dt\right)\right]\leq \exp\left(C_*^{-1}\int_0^T\log\left( \int e^{C_* h_t } d\mu_{X}^*\right)dt\right)\,,
\end{align}
where  $h_t\coloneqq h(t,\cdot)$.
\end{theorem}
\begin{proof} 

Start by letting let $h\in C^\infty_c(\mathbb{R}^{1+d})$, $f\in C^\infty_c(\mathbb{R}^d)$ (smooth functions with compact support). Existence of a $C^{1,2}$ classical solution to the 
\begin{align}\label{eq:Feynman_Kac_pde_app}
&\partial_t u^f_T(t,x)=-A_{X}[u^f_T(t,\cdot)](t,x)-h(t,x)  u^f_T(t,x)\,,\,\,\,\,u^f_T(T,x)=f(x)\,,
\end{align}
with $A_X$ as defined in \eqref{eq:gen_def}, such that  $u^f_T$ is bounded, and  $\partial_t u^f_T$, $D_x u^f_T$,  $D_x^2 u^f_T$ are all polynomially bounded in $x$ uniformly in $t\in[0,T]$ can be proven by PDE techniques; see Theorem 2.8 and Corollary 4.2 of \cite{krylov2009elliptic} (note that one can convert to the case where $\sup h<0$  by multiplying $u^f_T(t,x)$ by $e^{c(T-t)}$ for appropriate $c\in\mathbb{R}$).  Using the Feynman-Kac formula,   see,  e.g.,   Theorem 7.6 in Chapter 5 of  \cite{karatzas2014brownian}, we obtain a   representation    in terms of  $X^{t,x}_s$,  the solution to \eqref{eq:X} starting at $x\in\mathbb{R}^d$ at time $t$:
\begin{align}\label{eq:Feynman_Kac_formula_app}
u^f_T(t,x)= \mathbb{E}\left[ f(X^{t,x}_T)\exp\left(\int_t^T h(s, X_s^{t,x})ds\right)\right]\,.
\end{align}

Boundedness of $u^f_T$, polynomial boundedness of  $\partial_t u^f_T$, and  Assumptions \ref{assump:main_assumptions_1}\,.\,\ref{assump:ICS}  and \ref{assump:main_assumptions_2}\,.\,\ref{assump:inv_dist} allow us to use the dominated convergence theorem to compute
\begin{align}\label{eq:uf_deriv_lower_bound}
\frac{d}{dt} \|u^f_T(t,\cdot)\|^2_{L^2(\mu_{X}^*)}
=&2 \int u^f_T(t,x) \partial_tu^f_T(t,x)\mu_{X}^*(dx)\\
=& 2 \int u^f_T(t,x) (-A_X[u^f_T(t,\cdot)](x)-h(t,x)  u^f_T(t,x))\mu_{X}^*(dx) \,,\notag
\end{align} 
where we used the PDE \eqref{eq:Feynman_Kac_pde_app} to obtain the second line.

For all $t$, $u^f_T(t,\cdot)\in C^2_{pb}(\mathbb{R}^d)$ and therefore we can use Assumption \ref{assump:main_assumptions_2}\,.\,\ref{assump:log_Sob} to lower bound \eqref{eq:uf_deriv_lower_bound} as follows:
\begin{align}\label{eq:u_norm_deriv_lb}
&\frac{d}{dt} \|u^f_T(t,\cdot)\|^2_{L^2(\mu_{X}^*)}\\
\geq & -2\|u^f_T(t,\cdot)\|_{L^2(\mu_{X}^*)}^2\sup\left\{ \int   gA_X[g]d\mu_{X}^*+\int h_t g^2d\mu_{X}^*:g\in C^2_{pb}(\mathbb{R}^d),\|g\|_{L^2(\mu_{X}^*)}=1\right\}\notag\\
\geq&-2\|u^f_T(t,\cdot)\|_{L^2(\mu_{X}^*)}^2\sup\left\{ \int h_t g^2d\mu_{X}^*-C_*^{-1}\int g^2\log(g^2)d\mu_{X}^*:g\in C^2_{pb}(\mathbb{R}^d),\|g\|_{L^2(\mu_{X}^*)}=1\right\}\notag\\
\geq&-2C_*^{-1}\|u^f_T(t,\cdot)\|_{L^2(\mu_{X}^*)}^2\sup_{\nu:\mathrm{KL}(\nu\|\mu_{X}^*)<\infty}\left\{ \int C_* h_t  d\nu-\mathrm{KL}(\nu\|\mu_{X}^*)\right\}\notag\\
=&-2C_*^{-1}\|u^f_T(t,\cdot)\|_{L^2(\mu_{X}^*)}^2\log\left( \int e^{C_*h_t} d\mu_{X}^*\right)\,,\notag
\end{align}
 where the last line was obtained by using the formula for the convex conjugate of the KL divergence; see \cite[Proposition 4.5.1]{Dupuis_Ellis}.

Using the Markov property, the Cauchy-Schwarz inequality, the representation \eqref{eq:Feynman_Kac_formula_app}, and then Gr{\"o}nwall's inequality applied to \eqref{eq:u_norm_deriv_lb} with the terminal condition 
\begin{align}
\|u^f_T(T,\cdot)\|^2_{L^2(\mu_{X}^*)}=\|f\|^2_{L^2(\mu_{X}^*)}
\end{align}
 we obtain
\begin{align}
&  \mathbb{E}\left[ f(X_T)\exp\left(\int_0^T h(s, X_s )ds\right)\right]\\
=&\int \mathbb{E}\left[ f(X^{0,x}_T)\exp\left(\int_0^T h(s, X_s^{0,x})ds\right)\right]\mu_{X}^*(dx)\notag\\
\leq&\|u_T^f(0,\cdot)\|_{L^2(\mu_{X}^*)}\leq \|f\|_{L^2(\mu_{X}^*)}\exp\left(C_*^{-1}\int_0^T \log\left( \int e^{C_*h_t} d\mu_{X}^*\right)dt\right)\,.\notag
\end{align}
  Applying the above bound, which   holds for all $f\in C^\infty_c(\mathbb{R}^d)$, to a sequence of smooth bump functions $f_j$ that increase to $1$ and using the monotone convergence theorem to compute the limit as $j\to\infty$ we can conclude the claimed bound \eqref{eq:MGF_bound} for all $h\in C^\infty_c(\mathbb{R}^{1+d})$.

  Now suppose $h\in C_b(\mathbb{R}^{1+d})$ (bounded and continuous functions).   Construct a  sequence $h_j\in C^\infty_c(\mathbb{R}^{1+d})$ such that $\sup_j\|h_j\|_\infty<\infty$ and which converges pointwise to $h$.  Applying the above to each $h_j$ and using the dominated convergence theorem we obtain  \eqref{eq:MGF_bound} for all  $h\in C_b(\mathbb{R}^{1+d})$.

Finally, let  $h$ be as in Assumption \ref{assump:h_assump1} (continuously extended to $\mathbb{R}^{1+d}$).    Define $h_{\ell,m}\coloneqq- \ell 1_{h<-\ell}+h1_{-\ell\leq h\leq m}+m1_{h>m}$,  $\ell,m\in\mathbb{Z}^+$, and note that $h_{\ell,m}\in C_b(\mathbb{R}^{1+d})$,  $\lim_{\ell\to\infty}h_{\ell,m}=h_{\infty,m}\coloneqq h1_{h\leq m}+m1_{h>m}$  and $\lim_{m\to\infty}h_{\infty,m}=h$ pointwise.   In addition, we have the   $h_{\infty,m}\leq h_{\infty,m+1}$, $|h_{\ell,m}|\leq |h|$ and $|h_{\infty,m}|\leq |h|$. Using these together with  continuity of $t\mapsto X_t$, we are able to apply the dominated convergence theorem to obtain the pointwise limits
\begin{align}
&\lim_{m\to\infty}\int_0^T h_{\infty,m}(t,X_t)dt=\int_0^T h(t,X_t)dt\,,\\
&\lim_{\ell\to\infty}\int_0^Th_{\ell,m}(t,X_t) dt=\int_0^Th_{\infty,m}(t,X_t) dt\,.\notag
\end{align}
  Now apply the monotone convergence theorem and then the dominated convergence theorem to compute
\begin{align}
\mathbb{E}\left[\exp\left(\int_0^Th(t,X_t) dt\right)\right]
=&\lim_{m\to\infty}\mathbb{E}\left[\exp\left(\int_0^Th_{\infty,m}(t,X_t) dt\right)\right]\\
=&\lim_{m\to\infty}\lim_{\ell\to\infty}\mathbb{E}\left[\exp\left(\int_0^Th_{\ell,m}(t,X_t) dt\right)\right]\,.\notag
\end{align}
For all $\ell,m$ we have $h_{\ell,m}\in C_b(\mathbb{R}^{1+d})$, hence we can apply the previously proven case of this result to obtain
\begin{align}\label{eq:MGF_bound_double_limit}
\mathbb{E}\left[\exp\left(\int_0^Th(t,X_t) dt\right)\right]
\leq&\lim_{m\to\infty}\lim_{\ell\to\infty} \exp\left(C_*^{-1}\int_0^T\log\left( \int e^{C_* (h_{\ell,m})_t } d\mu_{X}^*\right)dt\right)\,.
\end{align}
Using Jensen's inequality together with the bounds  $|h_{\ell,m}|\leq |h|$  and $h_{\ell,m}\leq m$ we can compute
\begin{align}\label{eq:exp_int_lb}
-\int C_* |h_t|  d\mu_*\leq\log\left( \int e^{C_* (h_{\ell,m})_t } d\mu_{X}^*\right)\leq C_* m \,,
\end{align}
 where we note that both the upper and lower bounds are in $L^1([0,T],dt)$.  Hence we can apply the dominated convergence theorem to compute the inner limit in \eqref{eq:MGF_bound_double_limit}, thereby obtaining
\begin{align}
\mathbb{E}\left[\exp\left(\int_0^Th(t,X_t) dt\right)\right]
\leq&\liminf_{m\to\infty}  \exp\left(C_*^{-1}\int_0^T\log\left( \int e^{C_* (h_{\infty,m})_t } d\mu_{X}^*\right)dt\right)\,.
\end{align}
Finally, using the bound $h_{\infty,m}\leq h$ we arrive at   \eqref{eq:MGF_bound} in the general case; note that the right-hand side of  \eqref{eq:MGF_bound} is well-defined, though possibly equal to $+\infty$, due to the integrability of the lower bound in \eqref{eq:exp_int_lb}.
\end{proof}

%\bibliographystyle{siamplain}
%\bibliography{refs}
\bibliography{UQ_SDEs_Perturbed_Diffusion.bbl}

\end{document}